\documentclass{article}

\usepackage{amssymb,latexsym,amsmath,enumerate,verbatim,amsfonts,amsthm}
\usepackage{color,bm} 
\usepackage{graphicx}
\usepackage{algorithm,algorithmic}
\usepackage[colorinlistoftodos]{todonotes}

\newtheorem{theorem}{Theorem}[section]

\newtheorem{lemma}[theorem]{Lemma}
\newtheorem{proposition}[theorem]{Proposition}
\theoremstyle{definition}
\newtheorem{definition}[theorem]{Definition}

\numberwithin{equation}{section}
\allowdisplaybreaks[4]

\def\R{{\mathbb{R}}}

\def\argmin{\mathop{\rm arg\,min}}
\def\Argmin{\mathop{\rm Arg\,min}}

\def\Alg{{\sf ManPPA}_{\sf rw}}

\begin{document}
\makeatletter

\begin{center}
\large{\bf Convergence analysis for a variant of manifold proximal point algorithm based on Kurdyka-{\L}ojasiewicz property}\footnote{This work is supported partly by the Hong Kong Research Grants Council (Grant No. PolyU153002/21p) and the Natural Science Foundation of China (Grant No. 12201389)}
\end{center}
\begin{center}
Dedicated to Professor Terry Rockafellar on the occasion of his 90th birthday
\end{center}\vspace{3mm}
\begin{center}

\textsc{Peiran Yu, Liaoyuan Zeng and Ting Kei Pong\footnote{Corresponding author}}\end{center}

\vspace{2mm}
	
{\footnotesize
\noindent\begin{minipage}{14cm}
	{\bf Abstract:} We incorporate an iteratively reweighted strategy in the manifold proximal point algorithm (ManPPA) in \cite{ChDeMaSo20} to solve an enhanced sparsity inducing model for identifying sparse yet nonzero vectors in a given subspace. We establish the global convergence of the whole sequence generated by our algorithm by assuming the Kurdyka-{\L}ojasiewicz (KL) properties of suitable potential functions. We also study how the KL exponents of the different potential functions are related. More importantly, when our enhanced model and algorithm reduce, respectively, to the model and ManPPA with constant stepsize considered in \cite{ChDeMaSo20}, we show that the sequence generated converges linearly as long as the optimal value of the model is positive, and converges finitely when the limit of the sequence lies in a set of weak sharp minima. Our results improve \cite[Theorem~2.4]{ChDeMaSo20_2}, which asserts local quadratic convergence in the presence of weak sharp minima when the constant stepsize is small.
\end{minipage}
\\[5mm]
		
}

\hbox to 14cm{\hrulefill}\par
		
		
\section{Introduction}

The problem of identifying sparse yet nonzero vectors in a given subspace arises in many contemporary applications such as orthogonal dictionary learning \cite{BaiSun} and other problems in machine learning \cite{ChDeMaSo20,Zibulevsky}. This fundamental problem is NP-hard in general \cite{Coleman86}, and efforts have been directed at developing tractable approaches for solving it. In the seminal work \cite{QuSun14}, the authors considered the following nonconvex formulation
\begin{equation}\label{sparsevector-p2}
	\begin{array}{cl}
		\min\limits_{x\in \R^{n}} & \|Y^T x\|_1\\
		{\rm s.t.} & \|x\|=1,
	\end{array}
\end{equation}
where $Y\in \R^{n\times p}$, and $\|\cdot\|_1$ and $\|\cdot\|$ are the $\ell_1$ and $\ell_2$ norms, respectively. The ability of identifying sparse nonzero vectors upon solving \eqref{sparsevector-p2} was formally established in \cite[Theorem~1]{QuSun14} when the subspace spanned by the {\em rows} of $Y$ follows a certain planted sparse model. An algorithm based on alternating minimization was developed therein for solving \eqref{sparsevector-p2}, which is guaranteed to return globally optimal solutions with high probability under suitable assumptions; see \cite[Theorem~2]{QuSun14} for details. More recently, the manifold proximal point algorithm (ManPPA) was developed in \cite{ChDeMaSo20} for solving \eqref{sparsevector-p2}, with subsequential convergence to stationary points established in \cite[Theorem~5.5]{ChMaSoZh20}. In addition, it was shown in \cite[Theorem~2.4]{ChDeMaSo20_2} that if \eqref{sparsevector-p2} has a set of weak sharp minima ${\cal X}$, then as long as the initial point is chosen sufficiently close to ${\cal X}$, the sequence $\{x^k\}$ generated by ManPPA with a sufficiently small constant stepsize converges {\em quadratically}.

Here, we enhance the sparsity inducing model \eqref{sparsevector-p2} and analyze the convergence of a natural adaptation of the ManPPA to solve the enhanced model. Specifically, we consider the following problem with a more general sparsity inducing regularizer $\Phi$:
\begin{equation}\label{P0}
  \begin{array}{rl}
    \min\limits_{x\in\mathbb{R}^n} &\displaystyle f(x): = \Phi ( |Y^Tx | ) =  \sum\limits_{i=1}^p\phi ( |[Y^Tx]_i | )\\
    {\rm s.t.} & x\in {\cal M}:= \{u\in \R^n:\; \|u\|=1\},
    \end{array}
\end{equation}
where $Y \in\mathbb{R}^{n\times p}$, $ |y| $ is the vector whose $ i $th entry is $ |y_i| $, and $ \phi: \R_+\rightarrow \R_+ $ is a continuous concave function with $ \phi(0)=0 $, and is differentiable on $(0,\infty)$; in addition, the limit $\ell:= \lim_{s\downarrow 0}\phi'(s)$ exists.
Our assumptions on $\phi$ are general enough to cover many widely used \emph{nonconvex} sparsity inducing regularizers such as the smoothly clipped absolute deviation (SCAD) function \cite{FanLi01}, the minimax concave penalty (MCP) function \cite{Zhang10}, and the log penalty function \cite{CaWaBo08}. Now, given the separable structure of the function $\Phi$, it is natural to adapt the iterative reweighting strategy (see, e.g., \cite{CaWaBo08,ChaYin08}) for developing algorithms for solving \eqref{P0}. Our algorithm, which we call the manifold proximal point algorithm with reweighting ($\Alg$), can be seen as a natural extension of the ManPPA in \cite{ChDeMaSo20}. The key difference is that, while each subproblem of the ManPPA involves the $\ell_1$-norm in the objective, we derive our subproblem objectives by applying the iteratively reweighted $\ell_1$ strategy to the $\phi$ in the objective of \eqref{P0}, resulting in objectives involving \emph{weighted} $\ell_1$ norms; see Algorithm~\ref{alg1} in Section~\ref{sec2}. Notably, when $\phi(s) = s$, our model \eqref{P0} and algorithm $\Alg$ reduce to \eqref{sparsevector-p2} and the ManPPA in \cite{ChDeMaSo20} with constant stepsize, respectively.

In this paper, we study the global convergence properties of $\Alg$ for solving \eqref{P0}. We establish the global convergence of the \emph{whole} sequence generated by assuming the Kurdyka-{\L}ojasiewicz (KL) property of suitable potential functions, whose definitions depend on whether $\phi_+'$ is locally Lipschitz.\footnote{Here and throughout, $ \phi'_+(s):=\lim_{h\downarrow 0}\dfrac{\phi(s+h)-\phi(s)}{h} $.} We then study how the KL exponents of the different potential functions are related, and derive {\em explicit} convergence rate of $\Alg$ when $\phi(s) = s$ by computing the KL exponent of the potential function associated with \eqref{sparsevector-p2}. Specifically, we show that when the optimal value of \eqref{sparsevector-p2} is positive, the KL exponent of the associated potential function is $1/2$, which implies that the sequence $\{x^k\}$ generated by $\Alg$ converges linearly. Furthermore, when a limit point of $\{x^k\}$ does belong to a set of weak sharp minima, the convergence is indeed finite. Our results improve \cite[Theorem~2.4]{ChDeMaSo20_2}, which asserts local quadratic convergence for small stepsizes in the presence of weak sharp minima.

The rest of the paper is organized as follows. We present preliminary materials and our algorithm $\Alg$ in Section~\ref{sec2}. Global convergence of the whole sequence is studied in Section~\ref{sec3}, and Section~\ref{sec4} is devoted to the study of KL exponents.

\section{Preliminaries}\label{sec2}

In this paper, we use $\R^n$ to denote the $n$-dimensional Euclidean space with inner product $\langle \cdot,\cdot\rangle$ and Euclidean ($\ell_2$) norm $\|\cdot\|$, and $ \|\cdot\|_1 $ and $\|\cdot\|_\infty$ to represent the $ \ell_1 $ and $\ell_\infty$ norms, respectively. For $x$, $y$ in $\R^n$, we let $x\circ y $ denote the entrywise product, and $|x|$ denote the vector whose $i$th entry is $|x_i|$. For an $x\in \R^n$ and $r\ge 0$, we let $B(x,r) = \{u:\; \|u - x\|\le r\}$. For a matrix $ A\in \R^{m\times n} $, we use $ \|A\|_2 $ to denote its operator norm.

For the $ \Phi $ in \eqref{P0}, we define
\[
\Phi_+'(y)= (\phi_+'(y_1),\phi_+'(y_2),\ldots,\phi_+'(y_p))\in \R^p;
\]
note that $\Phi_+'(y)\in \R^p_+$ as a consequence of \cite[Lemma~2.2(i)]{YuPong17}. An extended-real-valued function $ h:\R^n\to (-\infty, \infty] $ is proper if it is finite for at least one point $ x\in\R^n $, and the domain of $ h $ is defined as $ {\rm dom}\,h:=\{x: h(x)<\infty\} $. A proper function is said to be closed if it is lower semi-continuous. For a proper closed function $ h $, its regular subdifferential $ \widehat{\partial} h $, limiting subdifferential $ \partial h $ and horizon subdifferential $\partial^\infty h$ at $\bar x\in {\rm dom}\,h $ are defined respectively as:
\begin{equation*}
\begin{split}
& \widehat{\partial}h(\bar x)= \left\{\upsilon:\; \liminf_{x\rightarrow \bar x, x\neq \bar x} \frac{h(x)-h(\bar x)-\langle\upsilon, x-\bar x\rangle}{\|x-\bar x\|}\geq 0\right\}, \\
& \partial h(\bar x)= \left\{\upsilon:\; \exists x^k\overset{h}\rightarrow \bar x \text{ and } \upsilon^k \in \widehat{\partial} h(x^k) \text{ with } \upsilon^k \rightarrow \upsilon\right\},\\
& \partial^\infty h(\bar x)= \left\{\upsilon:\; \exists x^k\overset{h}\rightarrow \bar x, \ t_k \downarrow 0 \text{ and } \upsilon^k \in \widehat{\partial} h(x^k) \text{ with } t_k\upsilon^k \rightarrow \upsilon\right\},
\end{split}
\end{equation*}
where $ x^k\overset{h}\rightarrow \bar x $ means $ x^k\rightarrow \bar x $ and $ h(x^k)\rightarrow h(\bar x) $. We set by convention that $ \widehat{\partial}h(x)=\partial h(x)=\partial^\infty h(x)=\emptyset $ when $ x\notin {\rm dom}\,h $, and we define $ {\rm dom}\,\partial h = \{x: \partial h(x)\neq\emptyset \} $. For a function $ h:\R^n\to (-\infty, \infty] $ that is locally Lipschitz around $ x\in {\rm dom}\,h $, we say that $ h$ is regular at $ x $ if $ \widehat{\partial} h(x) = \partial h(x) $. Moreover, we recall from \cite[Corollary~8.11]{RoWe97} that for a proper closed function $h$, regularity at an $x\in {\rm dom}\,\partial h$ can be characterized as $ \widehat{\partial} h(x) = \partial h(x) $ and $  \partial^\infty h(x)= (\widehat{\partial} h(x))^\infty $, where $(\widehat{\partial} h(x))^\infty$ is the recession cone of the closed convex set $\widehat{\partial} h(x)$.

For a set $C\subseteq \R^n$ and a vector $a\in \R^n$, we define $a\circ C = \{a\circ c:\; c\in C\}$. The indicator function of $C$ is given by
\begin{equation*}
	\delta_C(x)=
	\begin{cases}
	0 & {\rm if\ } x\in C, \\
	\infty & {\rm if\ } x\notin C,
	\end{cases}
\end{equation*}
and the normal cone of $ C $ at an $ x \in C$ is defined as $ N_C(x) = \partial \delta_C(x) $. We also let ${\rm dist}(x,C) = \inf_{y\in C}\|x - y\|$ denote the distance from an $x\in \R^n$ to $C$.

We next recall the definition of Kurdyka-{\L}ojasiewicz (KL) property and exponent of proper closed functions. The KL property plays a pivotal role in the convergence analysis of many contemporary first-order methods; see, e.g., \cite{Attouch10,Attouch13}.
\begin{definition}[{{Kurdyka-{\L}ojasiewicz property and exponent}}]
	We say that a proper closed function $h:\R^n\to(-\infty,\infty]$ satisfies the Kurdyka-{\L}ojasiewicz (KL) property at an $\widehat x\in {\rm dom}\, \partial h$ if there are $a\in (0,\infty]$, a neighborhood $V$ of $\widehat{x}$ and a continuous concave function $\varphi: [0,a)\rightarrow [0,\infty) $ with $\varphi(0)=0$ such that
	\begin{enumerate}[{\rm (i)}]
		\item $\varphi$ is continuously differentiable on $(0, a)$ with $\varphi'>0$ on $(0,a)$;
		\item for any $x\in V$ with $h(\widehat{x})<h(x)<h(\widehat{x})+a$, it holds that
		\begin{equation}\label{phichoice}
		\varphi'(h(x)-h(\widehat{x})){\rm dist}(0,\partial h(x))\ge 1.
		\end{equation}
	\end{enumerate}
	If $h$ satisfies the KL property at $\widehat x\in {\rm dom}\,\partial h$ and the $\varphi$ in \eqref{phichoice} can be chosen as $\varphi(\nu) = a_0\nu^{1-\alpha}$ for some $a_0>0$ and $\alpha\in[0,1)$, then we say that $h$ satisfies the KL property at $\widehat x$ with exponent $\alpha$.
	
	A proper closed function $h$ satisfying the KL property at every point in ${\rm dom}\, \partial h$ is called a KL function, and a proper closed function $h$ satisfying the KL property with exponent $\alpha\in [0,1)$ at every point in ${\rm dom}\, \partial h$ is called a KL function with exponent $\alpha$.
\end{definition}
KL functions are ubiquitous in contemporary applications. For example, proper closed semi-algebraic functions are KL functions with some exponent $ \alpha\in [0,1) $; see \cite[Corollary~16]{bolte2007}.

We next recall two other notions that will be used in our subsequent discussions: prox-regularity and weak sharp minima. Prox-regularity is a classical notion in variational analysis and we follow the definition in \cite{RoWe97}, and we refer the readers to \cite{Burke05} and references therein for more study on weak sharp minima of functions.
\begin{definition}[{Prox-regularity; see  \cite[Definition~13.27]{RoWe97}}]\label{DefPror}
	The function $h:\mathbb{R}^n\to \mathbb{R}\cup\{-\infty, \infty \}$ is prox-regular at $\bar x$ for $\bar v$ if $h$ is finite and locally lower semicontinuous at $\bar x$ with $\bar v\in\partial h(\bar x)$, and there exist $\epsilon>0$ and $r\ge 0$ such that
	\[
	h(x')\ge h(x) + \langle v,x' - x\rangle  - \frac{r}{2}\|x' - x\|^2 \ \ \ \ \forall x'\in B(\bar x,\epsilon)
	\]
	whenever $v\in\partial h(x)$,  $\|v - \bar v\|<\epsilon$, $\|x - \bar x\|<\epsilon$ and $h(x) <h(\bar x)+\epsilon$. 	
	We say that $h$ is prox-regular at $\bar x$ if $h$ is prox-regular at $\bar x$ for all $\bar v\in\partial h(\bar x)$.
\end{definition}
\begin{definition}[Weak sharp minima; see, e.g., \cite{BuFe93}]\label{sharpdef}
	We say that $\mathcal{X}\subseteq {\rm dom}\, h$ is a set of weak sharp minima for a proper function $h:\R^n\to (-\infty,\infty]$ with parameter $(\alpha,\delta)$ for some  $\alpha>0$ and $\delta>0$ if
	\[
	h(x)\ge h(\bar x) +  \alpha\, {\rm dist}(x,\mathcal{X})\ \ \ \ \  \forall \bar x\in\mathcal{X}\ {\rm and}\ \forall x\in
\mathfrak{B}(\delta): = \{u:\; {\rm dist}(u,\mathcal{X})\le \delta\}.
	\]
\end{definition}

One can see immediately from the definition of weak sharp minima that $h$ is constant on such a set.
Next, we present some properties of the objective function and feasible set in \eqref{P0} that will be used subsequently. The lemma below characterizes the subdifferential of $ f $ in \eqref{P0}.
\begin{lemma}\label{flip}
Let $ f $ be given in \eqref{P0}. Then $ f $ is Lipschitz continuous with modulus $ L_f :=\ell \sqrt{p}\|Y\|_2 $. Moreover, $ f $ is regular and $ \partial f(x) = Y\left(\Phi'_+(|Y^Tx|)\circ \partial\|Y^Tx\|_1\right) $.
\end{lemma}
\begin{proof}
 Since $ \phi $ is differentiable on $ (0,\infty) $ and continuous on $[0,\infty)$, for any $ x$, $y\in\R^n $, there exists $ \theta\in\R^p_+ $ such that
	\[
	\begin{split}
	& f(x)-f(y) = \sum\limits_{i=1}^p \left[ \phi ( |[ Y^Tx]_i | ) -  \phi ( |[ Y^Ty]_i | ) \right] = \sum\limits_{i=1}^p \left[\phi'_+(\theta_i) ( |[Y^Tx]_i | -  |[Y^Ty]_i | )\right] \\
	& \overset{\rm (a)}\leq \ell  \| |Y^Tx | -  |Y^Ty |  \|_1\leq \ell  \|Y^T(x-y) \|_1 \leq \ell\sqrt{p}\|Y^T(x-y)\| \leq \ell\sqrt{p}\|Y\|_2\|x-y\|,
	\end{split}
	\]
where we use \cite[Lemma~2.2(i)]{YuPong17} in (a). This proves the claim on Lipschitz continuity.

Next, following the arguments in \cite[Lemma~2.3]{YuPong17}, we see that $ \Phi(|\cdot|) $ is regular and $ \partial \Phi(|\cdot|)(y) = \Phi'_+(|y|)\circ \partial \|y\|_1 $. This together with \cite[Theorem~10.6]{RoWe97} gives
$\partial f(x) = Y\left(\Phi'_+(|Y^Tx|)\circ \partial\|Y^Tx\|_1\right)$.
\end{proof}

The following lemma gives some properties of the feasible set $ \mathcal M $ in \eqref{P0}, where item (iii) is a special case of \cite[Fact~3.6]{ChMaSoZh20} and its proof can be found in \cite{ChDeMaSo20_2}. We include the item here for ease of reference.
\begin{lemma}\label{deltaM}
Let $ {\cal M}=\{x\in \R^n:\;\|x\|=1\} $. Then the following statements hold.
\begin{enumerate}[{\rm (i)}]
	\item For any $ x\in\cal M $, $ \delta_{\cal M} $ is regular at $ x $ and $ \partial \delta_{\mathcal M}(x) = N_{\mathcal M}(x) = \{\lambda x:\;\lambda \in \R\} $.
	\item For any $ x\in\mathcal M $ and $ \zeta\in\R^n $, we have $ \inf_{y\in N_\mathcal M(x)}\|\zeta+y\| = \|(I-xx^T)\zeta\| $.
	\item For any $ x\in\cal M $ and $ d$ satisfying $\langle x,d\rangle = 0$, we have
	\begin{equation}\label{retrprop_000}
		\left\|\dfrac{x+d}{\|x+d\|} - x\right\| \leq \|d\|,
	\end{equation}
	and
	\begin{equation}\label{retrprop}
		\left\|\dfrac{x + d}{\|x+d\|}-(x+d) \right\|\leq \frac12 \|d\|^2.
	\end{equation}
\end{enumerate}
\end{lemma}
\begin{proof}
	According to \cite[Example~6.8]{RoWe97} and \cite[Exercise~8.14]{RoWe97}, we see that $ \delta_{\cal M} $ is regular at any $ x\in\cal M $ and $ \partial \delta_{\cal M}(x)=N_{\cal M}(x) = \{\lambda x:\; \lambda\in\R\} $. This proves item (i). Next, for any $x\in {\cal M}$ and $ \zeta\in\R^n $, we have
	\[
	\inf\limits_{y\in N_{\cal M}(x)} \|\zeta+y\| = \inf_{\lambda\in\R}\|\zeta+\lambda x\| = \left\|\zeta-\frac{x^T\zeta}{x^Tx}x\right\| = \|(I-xx^T)\zeta\|,
	\]
	where the last equality holds because $ \|x\| = 1 $. This proves item (ii). Finally, for item (iii), \eqref{retrprop_000} follows from \cite[Proposition~A.2]{ChDeMaSo20_2} and \eqref{retrprop} can be found in \cite[Proposition~A.1]{ChDeMaSo20_2}.
\end{proof}

We next recall the definition of stationary point of \eqref{P0}. Define
\begin{equation}\label{F}
F(x) = f(x) + \delta_\mathcal M(x).
\end{equation}
\begin{definition}[Stationary point]
A point $ x^*\in \R^n $ is called a {stationary point} of \eqref{P0} if $ 0\in \partial F(x^*) $.
\end{definition}
The next lemma concerns $ \partial F $.
\begin{lemma}\label{prop:partialF}
	Let $F$ be defined in \eqref{F}. Then for any $x\in {\rm dom}\,F$, it holds that
	\begin{equation}\label{partialF}
	\partial F(x) =  \partial f(x) + N_{\cal M}(x)=Y\left(\Phi'_+(|Y^Tx|)\circ \partial\|Y^Tx\|_1\right) + N_{\cal M}(x).
	\end{equation}
	Also, it holds that  $  {\rm dom}\,\partial F =\mathcal{M}$ and
	\begin{equation}\label{gjr0}
	{\rm dist}(0,\partial F(x)) = \inf_{\zeta\in\partial f(x)} \|(I-xx^T)\zeta\|\ \ \ \ \ \ \ \forall x\in\mathcal{M}.
	\end{equation}
\end{lemma}
\begin{proof}
The relation \eqref{partialF} is a direct consequence of \cite[Corollary~10.9]{RoWe97}, Lemma~\ref{flip} and Lemma~\ref{deltaM}(i). In particular, we  see that $  {\rm dom}\,\partial F =\mathcal{M}$. We also see from \eqref{partialF} that for any $x\in\mathcal{M}$,
\begin{equation*}
	{\rm dist}(0,\partial F(x)) = \inf\limits_{\zeta\in\partial f(x)}\inf\limits_{y\in N_{\cal M}(x)} \|\zeta+y\| = \inf\limits_{\zeta\in\partial f(x)} \|(I-xx^T)\zeta\|,
\end{equation*}
where the last equality follows from Lemma~\ref{deltaM}(ii).
\end{proof}

Before ending this section, we present our algorithm for \eqref{P0} as Algorithm~\ref{alg1} below. This algorithm is a natural extension of the ManPPA (with constant stepsize) proposed and studied in \cite{ChDeMaSo20}, which was designed for solving \eqref{P0} when $\phi(s) = s$. Here, to tackle \eqref{P0} with a general $\phi$, we incorporate the classical iteratively reweighted $\ell_1$ strategy (see, e.g., \cite{CaWaBo08,ChaYin08}) in our algorithm so that one needs to solve the subproblem \eqref{subp} that involves a weighted $\ell_1$ norm in each iteration. We thus call our algorithm the manifold proximal point algorithm with reweighting ($\Alg$). We would like to point out that when $ \phi(s) = s $ in \eqref{P0}, our $ \Alg $ reduces to the ManPPA in \cite{ChDeMaSo20} with constant stepsize, whose corresponding subproblem was solved by suitably adapting the semi-smooth Newton augmented Lagrangian method (SSN-ALM); see \cite[Section~III.B]{ChDeMaSo20}. One can observe that the SSN-ALM can also be adapted to solve the subproblem \eqref{subp}, and we refer the readers to \cite[Section~III.B]{ChDeMaSo20} for further discussions on this subproblem solver.
\begin{algorithm}
	{
	\caption{$\Alg$ for problem \eqref{P0}\label{alg1}}
	\begin{algorithmic}
		\STATE  Choose $x^0\in\mathcal{M}$ and $t \in (0,2/L_f)$, where $L_f$ is defined in Lemma~\ref{flip}. Let $k=0$.
		\begin{description}
			\item[\bf Step 1.] Set $ u^k =\Phi'_+(|Y^Tx^k|) $. Compute
			\begin{equation}\label{subp}
				\begin{array}{rcl}
					d^k = &\displaystyle \argmin\limits_{d\in \mathbb{R}^n} &\displaystyle\sum\limits_{i=1}^p u_i^k|[Y^T(x^k+d)]_i| + \dfrac{1}{2t}\|d\|^2\\[13 pt]
					&{\rm s.t.} & \langle x^k,d\rangle = 0.
				\end{array}
			\end{equation}			
			\item[\bf Step 2.] Set
			\begin{align}\label{xkplus1}
				x^{k+1} = \frac{x^k +  d^k}{\|x^k + d^k\|}.
			\end{align}
			Update $k \leftarrow k+1$ and go to {\bf Step 1}.
		\end{description}
	\end{algorithmic}
}
\end{algorithm}

\section{Global convergence of $\Alg$ under KL property}\label{sec3}
In this section, we study the convergence properties of the sequence generated by $\Alg$. We first establish the subsequential convergence of the sequence $\{x^k\}$ generated by $\Alg$. Then we discuss the global convergence of $\{x^k\}$ by considering the KL property of suitable potential functions, whose definitions depend on whether $\phi_+'$ is locally Lipschitz.

Theorem~\ref{sta} concerns subsequential convergence. Items~(i) and (iv) were proved for the case $ \phi(s)=s $ in \cite[Lemma~5.2]{ChMaSoZh20} and \cite[Theorem~5.5]{ChMaSoZh20}, respectively. Our result allows for more general nonconvex $\phi$.
\begin{theorem}[Subsequential convergence of $\Alg$]\label{sta}
	Let $F$ be defined in \eqref{F} and $\{(x^k,u^k,d^k)\}$ be generated by $\Alg$. Then the following statements hold.
	\begin{enumerate}[{\rm (i)}]
		\item For every $ k\geq 0 $, it holds that
\[
F(x^{k+1})\leq F(x^k) - \left(\dfrac1t - \dfrac{L_f}2\right)\|d^k\|^2.
\]
		\item It holds that $\lim_{k\to\infty} \|d^k\| = 0$.
		\item For every $k\ge 0$, there exist $\lambda_k\in\R$ and $\widetilde \xi^{k+1}\in \partial\|Y^T(x^k + d^k)\|_1$ such that
		\begin{align}\label{opt}
		0= \dfrac{1}{t}d^k + Y(u^k\circ\widetilde\xi^{k+1}) + \lambda_k x^k.
		\end{align}
		In addition, the sequence $\{\lambda_k\}$ is bounded.
		\item Any accumulation point of $\{x^k\}$ is a stationary point of \eqref{P0}.
	\end{enumerate}
\end{theorem}
\begin{proof}
 For every $ k\geq 0 $, notice that the objective function of \eqref{subp} is strongly convex with modulus $t^{-1}$, $ d=0 $ is feasible for \eqref{subp} and $ d^k $ is optimal for \eqref{subp}. Based on these observations, we can see that
	\begin{equation}
		\sum\limits_{i=1}^{p}u_i^k |[Y^T(x^k+d^k)]_i| \leq \sum\limits_{i=1}^{p}u_i^k |[Y^Tx^k]_i| - \frac{1}{t}\|d^k\|^2 \label{sc_subp}.
	\end{equation}
One can then deduce using the definitions of $ F $ and $ x^{k+1} $ and Lemma~\ref{flip} that
	\begin{equation*}
	\begin{split}
		& F(x^{k+1}) = f\left(\frac{x^k+ d^k}{\|x^k+ d^k\|}\right) \leq f(x^k+d^k) + L_f\left\|\frac{x^k+ d^k}{\|x^k+ d^k\|} - (x^k+d^k)\right\|\\
		& \overset{\rm (a)}\leq  f(x^k+d^k) + \frac{L_f}2\|d^k\|^2  = \sum_{i=1}^p \phi(|[Y^T(x^k+d^k)]_i|) + \frac{L_f}2\|d^k\|^2 \\
		& \overset{\rm (b)}\le F(x^k) + \sum_{i=1}^p u_i^k(|[Y^T(x^k+d^k)]_i| - |[Y^Tx^k]_i|) + \frac{L_f}2\|d^k\|^2 \\
		& \overset{\rm (c)}\leq F(x^k) - \frac1t \|d^k\|^2 + \frac{L_f}2\|d^k\|^2 = F(x^k) - \left(\dfrac1t - \dfrac{L_f}2\right)\|d^k\|^2,
	\end{split}
	\end{equation*}
	where (a) follows from \eqref{retrprop} with $x=x^k$ and $ d=d^k $, (b) follows from the concavity of $ \phi $ and the definition of $ u^k $, and we use \eqref{sc_subp} in (c). This proves item (i).
	
	Next, from item (i) and the assumption $ t < 2/{L_f} $, we see that $\{F(x^k)\}$ is nonincreasing. This together with $F(x^{k})\ge 0$ for all $k\ge0$ implies that $\lim_{k\to\infty}F(x^k) = \inf_k F(x^k)\ge 0$. Summing the inequality in item (i) from $k=0$ to $\infty$, we obtain
	\[
	\left(\dfrac1t - \dfrac{L_f}2\right)\sum_{k=0}^\infty\|d^k\|^2\le F(x^0) - \lim_{k\to\infty} F(x^k)\leq F(x^0)< \infty.
	\]
	We then have $ \lim_{k\to\infty}\|d^k\|=0 $. This proves item (ii).
	
	For item (iii), using the definition of $d^k$ as the minimizer of \eqref{subp} and \cite[Corollary~28.2.2]{Ro70}, we see that for every $ k \ge 0$, there exists $\lambda_k\in \R$ such that
	\[
	\begin{split}
	-\frac1t d^k -\lambda_kx^k &\in \partial \left(\sum\limits_{i=1}^p u^k_i\left|[Y^T(x^k+\cdot)]_i\right|\right)(d^k)  = \sum\limits_{i=1}^p \partial \left(u^k_i \left|[Y^T(x^k+\cdot)]_i\right|\right)(d^k) \\
	& = \sum\limits_{i=1}^p u^k_i \partial \left(\left|\left([Y^T(x^k+\cdot)]_i\right) \right|\right)(d^k) =  Y\left(u^k\circ \partial\|Y^T(x^k+d^k)\|_1\right),
	\end{split}
	\]
	where we recall that $ u_i^k = \Phi'_+(|[Y^Tx^k]_i|)\in\R^p_+ $ (see \cite[Lemma~2.2(i)]{YuPong17}), and the equalities follow from \cite[Theorem~23.8]{Ro70}.
	
	Now we show that $\{\lambda_k\}$ is bounded. Taking the inner product with $x^k$ on both sides  of \eqref{opt} and noting that $\|x^k\|=1$ and $\langle x^k,d^k\rangle = 0$, we obtain
\[
\lambda_k = - \langle {x^k}, Y(u^k\circ \widetilde\xi^{k+1})\rangle.
\]
Since $\widetilde \xi^{k+1}$ is a subgradient of the $\ell_1$ norm, we know that $\{\widetilde \xi^{k+1}\}$ is bounded. Moreover, we see from the definition of $ u^k $ and \cite[Lemma~2.2(i)]{YuPong17} that $ \{u^k\}\in [0, \ell]^p  $. These imply that $ \{Y(u^k\circ \widetilde\xi^{k+1})\} $ is bounded, and thus $\{\lambda_k\}$ is bounded.
	
Finally, we prove item (iv). Let $\widehat x$ be an accumulation point of $\{x^k\}$ and let $ \{x^{k_j}\} $ be a convergent subsequence with $ \lim_{j\rightarrow\infty} x^{k_j} = \widehat x $. Then $\widehat x\in {\cal M}$. Since the sequences $\{d^k\} $, $\{u^k\}$, $ \{\widetilde \xi^k\} $ and $ \{\lambda_k\} $ are bounded, by passing to further subsequences if necessary, we may assume without loss of generality that $ \lim_{j\rightarrow\infty}(d^{k_j},  \widetilde \xi^{k_j+1}, \lambda_{k_j}) = (\widehat d,  \widehat\xi, \widehat \lambda )$ for some $(\widehat d, \widehat\xi, \widehat \lambda )$ and $ \lim_{j\rightarrow \infty} u^{k_j} = \widehat u := \Phi_+^\prime(|Y^T\widehat{x}|) $. Passing to the limit as $ j\rightarrow \infty $ in \eqref{opt} with $k = k_j$, we obtain that
	\[
	0 = \frac 1t \widehat d + Y\left(\widehat u\circ \widehat \xi\right) + \widehat\lambda \widehat x = Y\left(\widehat u\circ \widehat \xi\right) + \widehat\lambda \widehat x,
	\]
where the second equality holds because $ \widehat d = 0 $ according to item (ii). In addition, note that
	\[
	\widehat\xi\in \partial \|Y^T(\widehat x+\widehat d)\|_1 = \partial \|Y^T\widehat x\|_1,
	\]
	where the inclusion follows from the definition of $ \{\widehat\xi^{k}\} $ and the closedness of $ \partial \|\cdot\|_1 $, and the equality holds because $ \widehat d = 0 $ according to item (ii). Combining the above two displays with \eqref{partialF} and Lemma~\ref{deltaM}(i), we conclude that $ 0\in \partial F(\widehat x) $, which means that $ \widehat x $ is a stationary point of \eqref{P0}.
\end{proof}


\subsection{Global convergence when $\phi_+'$ is locally Lipschitz}

While the $\{x^k\}$ generated by $\Alg$ is known to cluster at a stationary point of \eqref{P0} according to Theorem~\ref{sta}(iv), additional conditions are usually required for establishing the convergence of the whole sequence. Here, we discuss the global convergence of $\Alg$ by assuming that $\phi_+'$ is locally Lipschitz and that the $F$ in \eqref{F} satisfies the KL property. We start with the following lemma concerning bounds on the least norm element of $\partial F$.
\begin{lemma}\label{distbound}
	Let $F$ be defined in \eqref{F} and $\phi_+':\R_+\to \R_+$ be locally Lipschitz. Let $\{(x^k,d^k)\}$ be generated by $\Alg$. Then there exists $C>0$ such that
	\begin{equation*}
		{\rm dist}(0,\partial F(x^{k+1}))\le \left(C + \frac1t\right)\|d^k\|\ \ \ \ \  \forall k\ge 0.
	\end{equation*}
\end{lemma}
\begin{proof}
Since $\{x^k\}$ is bounded, we may assume that $\phi_+'$ is Lipschitz continuous with modulus $C_0 > 0$ on the convex closure of $\{|x^k_i|:\; k\ge 0,\, i = 1,\cdots,n\}$.
Notice that
	\begin{align*}
			{\rm dist}(0,\partial F(x^{k+1}))  &\overset{\rm (a)}\le\|Y(u^{k+1}\circ \widetilde \xi^{k+1}) + \lambda_k x^{k+1}\| \\
&= \|Y(u^k\circ \widetilde \xi^{k+1}) + \lambda_k x^{k} + \lambda_k(x^{k+1}  -x^k) + Y([u^{k+1} - u^k]\circ \widetilde \xi^{k+1})\|\\
&\overset{\rm (b)}= \left\|-\frac1td^k + \lambda_k(x^{k+1}  -x^k)+ Y([u^{k+1} - u^k]\circ \widetilde \xi^{k+1})\right\|\\
& \le \frac1t\|d^k\| + |\lambda_k|\|x^{k+1}-x^k\| + \|Y\|_2\|\widetilde \xi^{k+1}\|_\infty\|u^{k+1} - u^k\|\\
& \overset{{\rm (c)}}\le \frac1t\|d^k\| + |\lambda_k|\|x^{k+1}-x^k\| + \sqrt{p}C_0\|Y\|_2\|\|x^{k+1} - x^k\|\\
&\overset{{\rm (d)}}\le \left(\frac1t+ |\lambda_k| + \sqrt{p}C_0\|Y\|_2\right)\|d^k\|\le \left(\frac1t+ C\right)\|d^k\|,
		\end{align*}
where $ \widetilde\xi^{k+1}$ and $\lambda_k$ are given in \eqref{opt} and (a) holds thanks to Lemma~\ref{prop:partialF}, Lemma~\ref{deltaM}(i) and
	\[
	  \widetilde\xi^{k+1}\in\partial \|Y^T(x^k+d^k)\|_1 = \partial\left \|\frac{Y^T(x^k+d^k)}{\|x^k+d^k\|}\right\|_1 = \partial \|Y^Tx^{k+1}\|_1,
	\]
	(b) holds thanks to \eqref{opt}, (c) holds since $\|\widetilde \xi^{k+1}\|_\infty\le 1$ and $\phi_+'$ is Lipschitz continuous with modulus $C_0$ on the convex closure of $\{|x^k_i|:\; k\ge 0,\, i = 1,\cdots,n\}$, we use the definition of $ x^{k+1} $ and \eqref{retrprop_000} with $x=x^k$ and $ d= d^k $ in (d), and the last inequality holds for some constant $C > 0$ because $\{\lambda_k\}$ is bounded thanks to Theorem~\ref{sta}(iii).
\end{proof}

Based on Theorem~\ref{sta}(i) and the above lemma, it is now standard to prove the global convergence of $\{x^k\}$ following the framework in \cite{Attouch13,BolSabTeb14}, by assuming that the $F$ in \eqref{F} (which is continuous on its domain) satisfies the KL property. Specifically, the sequence $\{x^k\}$ is bounded as $\|x^k\|=1$ for all $k$. Using Theorem~\ref{sta}(i) and Lemma~\ref{distbound}, one can deduce the global convergence of $ \{x^k\} $ by following the proof in \cite[Theorem~2.9]{Attouch13} and \cite[Theorem~1]{BolSabTeb14}; particularly, one can show that $\{x^k\}$ is a Cauchy sequence by first proving that $\{\|d^k\|\}$ is summable and then invoking \eqref{retrprop_000} and the definition of $x^{k+1}$. The proof is standard and we omit it here.
\begin{theorem}\label{thm1}
	Suppose that $ F $ in \eqref{F} is a KL function with $\phi_+':\R_+\to \R_+$ being locally Lipschitz. Then the sequence $ \{x^k\} $ generated by $\Alg$ converges to a stationary point of \eqref{P0}.
\end{theorem}

\subsection{Global convergence for general $ \phi $}
In this subsection, we discuss the global convergence of $\Alg$ \emph{without} assuming $ \phi_+' $ to be locally Lipschitz, under suitable assumptions concerning KL property. Specifically, our analysis is based on the following potential function and its KL property:
\begin{equation}\label{Ftilde}
\widetilde{F}(x,u) = \Psi^*(-u) + \langle u,  |Y^T x | \rangle + \delta_\mathcal M(x),
\end{equation}
where $ \Psi(y):= \sum\limits_{i=1}^p\psi(y_i) $ and $ \Psi^*(w) := \sum\limits_{i=1}^p \psi^*(w_i) $ with $\psi$ being the convex function defined as
\begin{equation}\label{psi}
\psi(s) = \begin{cases}
-\phi(s) & \text{if}\; s \geq 0,\\
s^2- \ell s & \text{otherwise}
\end{cases}
\end{equation}
and $\psi^*$ being its conjugate,
and $\phi$, $Y$ and ${\cal M}$ are defined in \eqref{P0}; an analogous potential function was previously proposed and used in \cite[Section~5.2]{YePong20} for analyzing global convergence of the algorithm therein. Note that for each $x$, it holds that
\begin{equation}\label{infwF}
\inf_u \widetilde F(x,u) = F(x)
\end{equation}
thanks to \cite[Theorem~12.2]{Ro70} and the definitions of $ \widetilde F $ and Fenchel conjugate, where $F$ is defined in \eqref{F}.

We first derive a formula for the subdifferential of $\widetilde F$ in \eqref{Ftilde}. We will start with an auxiliary lemma, which concerns the subdifferential of the following function $H$ extracted from the second summand on the right-hand side of \eqref{Ftilde}:
\begin{align}\label{defqH}
    H(x,u) :=  \langle u,  |Y^T x| \rangle = \langle u,q(x)\rangle,\ \ \ \ \ \ {\rm where }\ q(x) :=   |Y^T x |.
\end{align}
\begin{lemma}\label{subdiffH}
  Let $H$ be defined in \eqref{defqH}. For any $u\in \R^p_+$ and $x\in \R^n$, it holds that
  \begin{equation*}
  \partial H(x,u) = \begin{bmatrix}
    u\circ \partial q(x)\\ q(x)
  \end{bmatrix} = \widehat\partial H(x,u),
  \end{equation*}
  where we define $\partial q(x) := \partial q_1(x)\times \cdots \times\partial q_p(x)$ for notational simplicity.
\end{lemma}
\begin{proof}
  For any $(\bar x,\bar u)\in \R^n\times \R^p_+$, define the auxiliary function
  \begin{align*}
    H_{(\bar x,\bar u)}(x,u):= \langle \bar u,q(x)\rangle + \langle u,q(\bar x)\rangle.
  \end{align*}
  Then we have in view of \cite[Corollary~1.111(i)]{Morduk06} that
  \begin{align}\label{hahaha}
  \partial H(\bar x,\bar u) = \partial H_{(\bar x,\bar u)}(\bar x,\bar u) =
  \begin{bmatrix}
    \bar u\circ \partial q(\bar x)\\ q(\bar x)
  \end{bmatrix},
  \end{align}
  where the last equality follows from \cite[Proposition~10.5]{RoWe97}, the fact that $\bar u\in \R^p_+$, the convexity and finite-valuedness of $q_i$ for each $i$, and \cite[Theorem~23.8]{Ro70}. In addition, notice that $ H_{(\bar x,\bar u)}$ is convex and hence $\partial H_{(\bar x,\bar u)}(\bar x,\bar u) = \widehat \partial H_{(\bar x,\bar u)}(\bar x,\bar u)$ in view of \cite[Proposition~8.12]{RoWe97}. Thus, the desired conclusion will follow from this last observation and \eqref{hahaha} once we can show $\widehat \partial H_{(\bar x,\bar u)}(\bar x,\bar u)\subseteq \widehat \partial H(\bar x,\bar u)$, because we always have $\widehat \partial H(\bar x,\bar u)\subseteq \partial H(\bar x,\bar u)$ by the definitions of subdifferentials.

  It now remains to show $\widehat \partial H_{(\bar x,\bar u)}(\bar x,\bar u)\subseteq \widehat \partial H(\bar x,\bar u)$. To this end, pick any $(\xi,\eta)\in \widehat \partial H_{(\bar x,\bar u)}(\bar x,\bar u)$. Then using the definition of regular subdifferential, we see that for all $(x,u)$ close to $(\bar x,\bar u)$, we have
  \begin{align*}
    &o(\|(x,u) - (\bar x,\bar u)\|) \le H_{(\bar x,\bar u)}(x,u) - H_{(\bar x,\bar u)}(\bar x,\bar u)
    -\langle (\xi,\eta),(x,u) - (\bar x,\bar u)\rangle\\
    &\ \ \ \ \ \ \ \ \ \ \ \  = \langle \bar u,q(x)\rangle + \langle u,q(\bar x)\rangle - 2\langle \bar u,q(\bar x)\rangle-\langle (\xi,\eta),(x,u) - (\bar x,\bar u)\rangle\\
    &\ \ \ \  \ \ \ \ \ \ \ \  = \langle u,q(x)\rangle -  \langle \bar u,q(\bar x)\rangle + \langle \bar u - u, q(x) - q(\bar x)\rangle-\langle (\xi,\eta),(x,u) - (\bar x,\bar u)\rangle\\
    &\ \ \ \ \ \ \ \ \ \ \ \  \le \langle u,q(x)\rangle -  \langle \bar u,q(\bar x)\rangle -\langle (\xi,\eta),(x,u) - (\bar x,\bar u)\rangle
    + \|Y\|_2\|u - \bar u\|\|x - \bar x\|\\
    &\ \ \ \ \ \ \ \ \ \ \ \  \le \langle u,q(x)\rangle -  \langle \bar u,q(\bar x)\rangle -\langle (\xi,\eta),(x,u) - (\bar x,\bar u)\rangle
    + \frac{\|Y\|_2}2(\|u - \bar u\|^2 + \|x - \bar x\|^2).
  \end{align*}
  Consequently, $(\xi,\eta)\in \widehat \partial H(\bar x,\bar u)$.
\end{proof}

Now we are ready to present a formula for the subdifferential of $\widetilde F$ in \eqref{Ftilde} in the next proposition.
\begin{proposition}\label{subdifftildeF}
  Let $\widetilde F$ be given in \eqref{Ftilde}. Then for any $(x, u)\in {\rm dom}\,\widetilde F$, it holds that $u\in \R^p_+$ and
  \[
  \partial \widetilde F(x,u) = \begin{bmatrix}
    Y(u\circ \partial \|Y^Tx\|_1) + \partial \delta_{\cal M}(x)\\
    |Y^Tx| - \partial \Psi^*(-u)
  \end{bmatrix}.
  \]
\end{proposition}
\begin{proof}
  For any $(\bar x,\bar u)\in {\rm dom}\,\widetilde F$, we have $-\bar u \in {\rm dom}\,\Psi^* \subseteq \overline{\bigcup_{y\in \R^p}\partial \Psi(y)} \subseteq \R^p_-$, where the first set inclusion follows from the Br{\o}ndsted-Rockafellar Theorem (see, e.g., \cite[Theorem~3.1.2]{Za02}), and the second set inclusion follows from the definition of $\psi$ in \eqref{psi} and \cite[Lemma~2.2(i)]{YuPong17}. Thus, it holds that $\bar u\in \R^p_+$.

  Next, notice that $\widetilde F(x,u) = \Psi^*(-u) +H(x,u) + \delta_\mathcal M(x)$, where $H$ is defined in \eqref{defqH}. Since $H$ is clearly locally Lipschitz, we have from \cite[Corollary~10.9]{RoWe97} and \cite[Proposition~10.5]{RoWe97} that
 \begin{align}\label{eq1}
   \partial \widetilde F(\bar x,\bar u) \subseteq \partial H(\bar x,\bar u) + \begin{bmatrix}
     \partial \delta_{\cal M}(\bar x)\\
     -\partial \Psi^*(-\bar u)
   \end{bmatrix} = \begin{bmatrix}
    Y(\bar u\circ \partial \|Y^T\bar x\|_1) + \partial \delta_{\cal M}(\bar x)\\
    |Y^T\bar x| - \partial \Psi^*(-\bar u)
  \end{bmatrix},
 \end{align}
 where the equality follows from Lemma~\ref{subdiffH} and the fact that $\bar u \in \R^p_+$. In addition, we have
 \begin{align}\label{eq2}
   \begin{bmatrix}
    Y(\bar u\circ \partial \|Y^T\bar x\|_1) + \partial \delta_{\cal M}(\bar x)\\
    |Y^T\bar x| - \partial \Psi^*(-\bar u)
  \end{bmatrix} &= \partial H(\bar x,\bar u) + \begin{bmatrix}
     \partial \delta_{\cal M}(\bar x)\\
     -\partial \Psi^*(-\bar u)
   \end{bmatrix} \notag\\
   &= \widehat \partial H(\bar x,\bar u) + \begin{bmatrix}
     \widehat \partial \delta_{\cal M}(\bar x)\\
     -\widehat \partial \Psi^*(-\bar u)
   \end{bmatrix} \subseteq \widehat \partial \widetilde F(\bar x,\bar u),
 \end{align}
 where the second equality follows from Lemma~\ref{subdiffH}, the fact that $\bar u \in \R^p_+$, Lemma~\ref{deltaM}(i) and \cite[Proposition~8.12]{RoWe97}, and the inclusion is due to \cite[Corollary~10.9]{RoWe97} and \cite[Proposition~10.5]{RoWe97}. Since we always have $\widehat \partial \widetilde F(\bar x,\bar u)\subseteq \partial \widetilde F(\bar x,\bar u)$ by the definitions of subdifferentials, the desired result follows upon combining \eqref{eq1} and \eqref{eq2}.
\end{proof}

We now start our convergence analysis for the global convergence of the sequence generated by $\Alg$. We first prove a bound on the least norm element of $\partial \widetilde F$ along a \emph{suitable sequence} generated by $\Alg$.

\begin{lemma}\label{boundedrFtilde}
	Let $\widetilde F$ be defined in \eqref{Ftilde} and $\{(x^k,u^k, d^k)\}$ be generated by $\Alg$. Then $(x^{k+1},u^k)\in {\rm dom}\, \widetilde F$ for all $k\ge 0$, and there exists $C>0$ such that
	\begin{equation*}
	{\rm dist}(0,\partial \widetilde F(x^{k+1}, u^{k}))\le \left(C + \frac1t\right)\|d^k\|\ \ \ \ \ \ \ \forall k\ge 0.
	\end{equation*}
\end{lemma}
\begin{proof}
	For each $k \ge 0$, observe that $u^k = \Phi'_+(|Y^Tx^k|) = -\Psi'(|Y^Tx^k|)$. Thus, we see from \cite[Theorem~23.5]{Ro70} that
\begin{equation}\label{subdifferentialhaha}
|Y^Tx^k| \in \partial \Psi^*(-u^k),
\end{equation}
which in particular implies that $-u^k\in {\rm dom}\,\Psi^*$. Since $x^k\in {\cal M}$ for all $k\ge 1$ due to \eqref{xkplus1}, we conclude that $(x^{k+1},u^k)\in {\rm dom}\, \widetilde F$ for all $k\ge 0$.
Thus, we can invoke Proposition~\ref{subdifftildeF} to get for each $k \ge 0$ that
\begin{equation*}
\!\!\!\!\partial \widetilde F(x^{k+1}\!,u^{k}) \!=\! \begin{bmatrix}
    Y(u^{k}\circ \partial \|Y^Tx^{k+1}\|_1) \!+\! \partial \delta_{\cal M}(x^{k+1})\\
    |Y^Tx^{k+1}| - \partial \Psi^*(-u^{k})
  \end{bmatrix} \!\!\supseteq\!\! \begin{bmatrix}
    Y(u^{k}\circ \partial \|Y^Tx^{k+1}\|_1) \!+\! \partial \delta_{\cal M}(x^{k+1})\\
    |Y^Tx^{k+1}| - |Y^Tx^{k}|
  \end{bmatrix},\!\!\!\!
\end{equation*}
where the inclusion follows from \eqref{subdifferentialhaha}.

Then we have for every $ k\geq 0 $ that
	\begin{equation*}
	\begin{split}
	&{\rm dist}(0, \partial \widetilde F(x^{k+1}, u^{k}))\\
&\leq {\rm dist}(0, Y(u^{k}\circ \partial \|Y^Tx^{k+1}\|_1) + \partial \delta_{\cal M}(x^{k+1})) + \||Y^Tx^{k+1}| - |Y^Tx^k|\| \\
	&\overset{\rm (a)}\leq \|Y(u^{k}\circ \widetilde\xi^{k+1}) +\lambda_k x^{k+1}\| + \|Y^T\|_2\|x^{k+1} - x^k\| \\
	&= \|Y(u^k\circ \widetilde \xi^{k+1}) + \lambda_k x^k + \lambda_k(x^{k+1}-x^k)\| + \|Y^T\|_2\|x^{k+1} - x^k\|\\
	& \overset{\rm (b)}\le \frac 1t \|d^k\| +|\lambda_k|\cdot\|x^{k+1}-x^k\| + \|Y^T\|_2\|x^{k+1} - x^k\| \overset{\rm (c)}\leq  (\|Y\|_2+|\lambda_k|+\frac 1t)\|d^k\|,		
	\end{split}
	\end{equation*}
	where $ \widetilde\xi^{k+1}$ and $\lambda_k$ are given in \eqref{opt} and (a) holds thanks to Lemma~\ref{deltaM}(i) and the observation that
	\[
	  \widetilde\xi^{k+1}\in\partial \|Y^T(x^k+d^k)\|_1 = \partial\left \|\frac{Y^T(x^k+d^k)}{\|x^k+d^k\|}\right\|_1 = \partial \|Y^Tx^{k+1}\|_1,
	\]
	and (b) holds thanks to \eqref{opt}, and (c) uses the definition of $ x^{k+1} $ and \eqref{retrprop_000} with $x=x^k$ and $ d= d^k $. The desired conclusion follows upon noting that $\{\lambda_k\}$ is bounded, thanks to Theorem~\ref{sta}(iii).
\end{proof}

Next, we derive a descent property analogous to Theorem~\ref{sta}(i), but for $\widetilde F$ in \eqref{Ftilde} instead of the $F$ in \eqref{F}.
\begin{lemma}\label{boundedrFtilde2}
	Let $\widetilde F$ be defined in \eqref{Ftilde} and $\{(x^k,u^k, d^k)\}$ be generated by $\Alg$. Then for all $k\ge 1$, we have
\[
\widetilde F(x^{k+1},u^k) \le \widetilde F(x^{k},u^{k-1})- \left(\frac1t-\frac{L_f}{2}\right)\|d^k\|^2.
\]
\end{lemma}
\begin{proof}
  For every $ k\geq 0 $, using the strong convexity of the objective of \eqref{subp} (with modulus $t^{-1}$) and the feasibility of $d = 0$ for \eqref{subp}, we have from the optimality of $d^k$ that
	\begin{equation}
		\sum\limits_{i=1}^{p}u_i^k |[Y^T(x^k+d^k)]_i| \leq \sum\limits_{i=1}^{p}u_i^k |[Y^Tx^k]_i| - \frac{1}{t}\|d^k\|^2 \label{sc_subp2}.
	\end{equation}
In addition, note from Lemma~\ref{boundedrFtilde} that $(x^{k+1},u^k)\in {\rm dom}\,\widetilde F$. Then we have for all $k\ge 1$ that
	\begin{align*}
		& \widetilde F(x^{k+1},u^k) = \Psi^*(-u^k) +\langle u^k, |Y^T x^{k+1}|\rangle \\
& = \Psi^*(-u^k)  +\langle u^k, |Y^T (x^k + d^k)|\rangle+\langle u^k, |Y^T x^{k+1}| - |Y^T(x^k + d^k)|\rangle\\
& \le \Psi^*(-u^k)  +\langle u^k, |Y^T (x^k + d^k)|\rangle+\|u^k\|\|Y^T\|_2\left\|\frac{x^k+ d^k}{\|x^k+ d^k\|} - (x^k+d^k)\right\|\\
& \overset{\rm (a)}\le \Psi^*(-u^k)  +\langle u^k, |Y^T x^k|\rangle - \frac1t\|d^k\|^2 +\|u^k\|\|Y^T\|_2\left\|\frac{x^k+ d^k}{\|x^k+ d^k\|} - (x^k+d^k)\right\|\\
		& \overset{\rm (b)}\leq  \Psi^*(-u^k)  +\langle u^k, |Y^T x^{k}|\rangle + \left(\frac{\sqrt{p}\ell\|Y^T\|_2}{2}-\frac1t\right)\|d^k\|^2
\overset{\rm (c)}= -\Psi(|Y^T x^{k}|) + \left(\frac{L_f}{2}-\frac1t\right)\|d^k\|^2   \\
		& \overset{\rm (d)}\le \Psi^*(-u^{k-1}) +\langle u^{k-1}, |Y^T x^{k}|\rangle + \left(\frac{L_f}{2}-\frac1t\right)\|d^k\|^2  = \widetilde F(x^{k},u^{k-1})+ \left(\frac{L_f}{2}-\frac1t\right)\|d^k\|^2,
	\end{align*}
	where (a) follows from \eqref{sc_subp2}, (b) follows from \eqref{retrprop} (with $x = x^k$ and $ d=d^k $) and the fact that $|u_i^k|\le \ell$ (see \cite[Lemma~2.2(i)]{YuPong17}), (c) holds thanks to the fact that $u^k = \Phi'_+(|Y^Tx^k|) = -\Psi'(|Y^Tx^k|) \in -\partial\Psi(|Y^Tx^k|)$ and the equality case for the Young's inequality (see \cite[Theorem~23.5]{Ro70}), and we also used the definition of $L_f$ in Lemma~\ref{flip}, (d) follows from another application of the Young's inequality.
\end{proof}

Assuming that $ \widetilde F $ in \eqref{Ftilde} satisfies the KL property, we are now able to establish the global convergence of $\{x^k\}$, following the framework established in \cite{Attouch13,BolSabTeb14}. Indeed, since $x^k\in\mathcal{M}$ for all $k$, we see that $\{x^k\}$ is bounded. In addition, $\{u^k\}$ is also bounded because $\phi'_+$ is continuous. Moreover, in view of Lemma~\ref{boundedrFtilde2} and \eqref{infwF}, we know that $\{\widetilde F(x^k,u^{k-1})\}$ is nonincreasing and bounded below by $0$; hence, the limit $\lim_{k\to \infty}\widetilde F(x^k,u^{k-1})=:\tau$ exists. Now, if we let $\Xi$ denote the set of accumulation points of $\{(x^k,u^{k-1})\}$, then one can show that
$\widetilde F \equiv \tau$ on $\Xi$.\footnote{To see this, pick any $(\hat x,\hat u)\in \Xi$. Then we have from the lower semicontinuity of $\widetilde F$ that $\widetilde F(\hat x,\hat u)\le \tau$. As for the converse inequality, let $\{(x^{k_i+1},u^{k_i})\}$ be a subsequence converging to $(\hat x,\hat u)$. Then
\begin{align*}
  \tau &= \lim_{k\to \infty}\widetilde F(x^k,u^{k-1}) = \lim_{i\to \infty}\widetilde F(x^{k_i+1},u^{k_i}) = \lim_{i\to \infty} \Psi^*(-u^{k_i}) +\langle u^{k_i}, |Y^T x^{k_i+1}|\rangle\\
  & \overset{\rm (a)}= \lim_{i\to \infty} -\Psi(|Y^Tx^{k_i}|) +\langle u^{k_i}, |Y^T x^{k_i+1}| - |Y^Tx^{k_i}|\rangle \overset{\rm (b)}= -\Psi(|Y^T\hat x|) = F(\hat x)\le \widetilde F(\hat x,\hat u),
\end{align*}
where (a) follows from the definition of $u^{k_i}$ and the equality case for the Young's inequality (see \cite[Theorem~23.5]{Ro70}), (b) holds because Lemma~\ref{sta}(ii) and \eqref{retrprop_000} (with $x=x^{k_i}$ and $d=d^{k_i}$) imply that $\lim_{i\to\infty}\|x^{k_i+1} - x^{k_i}\| = 0$, and the last inequality follows from \eqref{infwF}.}
Using the above observations together with Lemma \ref{boundedrFtilde} and Lemma \ref{boundedrFtilde2}, one can deduce the global convergence of $ \{x^k\} $ by following the proofs in \cite[Theorem~2.9]{Attouch13} and \cite[Theorem~3.1]{LiPoTa19}; specifically, one can first deduce that $\{\|d^k\|\}$ is summable, which together with \eqref{retrprop_000} and the definition of $x^{k+1}$ shows that $\{\|x^{k+1}-x^k\|\}$ is summable, implying that $\{x^k\}$ is a Cauchy sequence. The proof is standard and we omit it here.
\begin{theorem}\label{prop1}
Suppose that $ \widetilde F $ in \eqref{Ftilde} is a KL function. Then the sequence $ \{x^k\} $ generated by $\Alg$ converges to a stationary point of \eqref{P0}.
\end{theorem}

\section{Convergence rate and KL exponents}\label{sec4}

In Theorem~\ref{sta}, we established subsequential convergence of the sequence $\{x^k\}$ generated by $\Alg$. We then studied the global convergence of the whole sequence by considering the KL property of suitable potential functions, namely, \eqref{F} and \eqref{Ftilde}, in Theorems~\ref{thm1} and \ref{prop1}, respectively. In this section, we are interested in studying the KL exponent of these potential functions: it is known that the KL exponents of potential functions used for global convergence analysis are generally closely related to the local convergence rate of the sequence generated by the algorithm; see, e.g., \cite{Attouch10,LiPo18}.

\subsection{KL exponent of $F$ in \eqref{F} when $\phi(s)=s$}
When $\phi(s) = s$, problem \eqref{P0} reduces to
\begin{equation}\label{Manpp}
	\begin{array}{rl}
		\min\limits_{x\in\mathbb{R}^n} &\widehat f(x):=\left\|Y^Tx\right\|_1\\
		{\rm s.t.} & x\in {\cal M},
	\end{array}
\end{equation}
and our $\Alg$ reduces to the ManPPA in \cite{ChDeMaSo20} with constant stepsize.
As established in Theorem~\ref{thm1}, the global convergence of the sequence $\{x^k\}$ generated by $\Alg$ (with $\phi(s)=s$) for \eqref{Manpp} was proved based on the KL property of the function $\widehat F$ defined as
\begin{equation}\label{F1}
	\widehat F(x) = \|Y^Tx\|_1+\delta_{\cal M}(x)= \widehat{f}(x) + \delta_{\cal M} (x),
\end{equation}
which is obtained by setting $\phi(s)=s$ in $F$ in \eqref{F}. In this subsection, we study the convergence rate of $\Alg$ (and hence the ManPPA) for \eqref{Manpp}.

In \cite[Theorem~2.4]{ChDeMaSo20_2}, it was shown that if \eqref{Manpp} has a set of weak sharp minima $\mathcal X$, then there exist $C>0$, $\delta > 0$ and a sufficiently small stepsize $t$ (which is less than $2/L_f$ according to the line below \cite[Eq.~(C.6)]{ChDeMaSo20_2}) such that if ${\rm dist}(x^0,{\cal X})\le \delta$, then the sequence $\{x^k\}$ generated by ManPPA with constant stepsize satisfies
\[
{\rm dist}(x^{k+1},\mathcal{X})\le C{\rm dist}^2(x^k,\mathcal{X})
\]
for all large $k$. Here, we show that, \emph{without} the weak sharp minima assumption, the sequence $\{x^k\}$ generated by $\Alg$ converges linearly as long as the optimal value of \eqref{Manpp} is nonzero; moreover, if $\lim_{k\to\infty}x^k$ belongs to a set of weak sharp minima, then $\{x^k\}$ converges finitely. Our results improve the claims in \cite[Theorem~2.4]{ChDeMaSo20_2}. The key to these results is to estimate the KL exponent of $ \widehat F $ in \eqref{F1}. We present these results in the next two subsections.

\subsubsection{Linear convergence of $\Alg$ for \eqref{Manpp}}

We will establish the linear convergence of $\Alg$ for \eqref{Manpp} with nonzero optimal values by explicitly computing the KL exponent of \eqref{F1}.
We start with the following auxiliary result.
\begin{lemma}
	Let $\widehat{F}$ be defined in \eqref{F1} and $\bar x\in \mathcal {M}$. Define
	\begin{align}\label{G}
	G_{\bar x}(x)=\widehat{f}\left(\frac{x}{\|x\|}\right)+\delta_{B\left(\bar x,1/2\right)}(x),
	\end{align}
	where $\widehat{f}$ is defined in \eqref{Manpp}. Then
	\begin{align}\label{partialG}
	\partial G_{\bar x}(x) = \frac{1}{\|x\|}\left(I - \frac{x}{\|x\|}\left(\frac{x}{\|x\|}\right)^T\right) \partial \widehat{f}\left(\frac{x}{\|x\|}\right)\ \ \ \ \ \ \forall x\in {\rm int}(B(\bar x,1/2)).
	\end{align}
\end{lemma}
\begin{proof}
Note that $G_{\bar x}(x) = \widehat f(x/\|x\|)$ for any $x\in {\rm int}(B(\bar x,1/2))$. In addition, since $\widehat f$ is convex, we know from \cite[Proposition~8.12]{RoWe97} that $\widehat f$ is regular. The desired conclusion now follows immediately from these observations and
\cite[Theorem 10.6]{RoWe97}.
\end{proof}
We next relate the KL property of $\widehat{F}$ in \eqref{F1} at $ \bar x\in\mathcal M $ to that of $ G_{\bar x} $ in \eqref{G}. Notice that
\begin{equation}\label{domains}
  {\rm dom}\,\partial \widehat F = {\cal M}\ \ \ \ {\rm and}\ \ \ \ {\rm dom}\,\partial G_{\bar x} = {\rm dom}\,G_{\bar x} = B(\bar x,1/2),
\end{equation}
where the first relation comes from Lemma~\ref{prop:partialF} and the second relation follows from \cite[Corollary~10.9]{RoWe97}.
\begin{lemma}\label{lem1}
Let $\widehat{F}$ be defined in \eqref{F1} and $\bar x\in {\rm dom}\,\partial\widehat F$. Let $G_{\bar x}$ be defined in \eqref{G}. Then the following statements hold.
\begin{enumerate}[{\rm (i) }]
	\item If $ \widehat{F} $ satisfies the KL property at $ \bar x $ with exponent $ \alpha\in(0,1) $, then  $ G_{\bar x}$ satisfies the KL property at $ \bar x $ with exponent $ \alpha $.
	\item If $ G_{\bar x}$ satisfies the KL property at $ \bar x $ with exponent $ \alpha\in(0,1) $, then $ \widehat{F} $ satisfies the KL property at $ \bar x $ with exponent $ \alpha $.
\end{enumerate}
\end{lemma}
\begin{proof}
	For any $ x \in {\rm int}(B(\bar x, 1/2))$, using \eqref{partialG}, it holds that
	\begin{equation}\label{distFG}
	\begin{split}
	& {\rm dist}(0, \partial G_{\bar x}(x))  = \frac{1}{\|x\|}\inf\limits_{\zeta\in\partial \widehat{f}\left(\frac{x}{\|x\|}\right)} \left\|\left(I- \frac{x}{\|x\|}\left(\frac{ x}{\|x\|}\right)^T\right)\zeta\right\|\\
	&\geq \frac{2}{3} \inf\limits_{\zeta\in\partial \widehat{f}\left(\frac{x}{\|x\|}\right)} \left\|\left(I- \frac{x}{\|x\|}\left(\frac{ x}{\|x\|}\right)^T\right)\zeta\right\|,
	\end{split}
	\end{equation}
	where the inequality holds because $  x\in B(\bar x, 1/2) $ and thus $ \|x\| \leq 3/2 $.
	Since $ \widehat{F} $ satisfies the KL property at $ \bar x \in {\rm dom}\,\partial \widehat F = {\cal M}$ (see \eqref{domains}) with exponent $ \alpha $ and is continuous on its domain, there exist $ c_1>0 $ and $ \epsilon_1>0 $ such that
	\begin{equation}\label{gjr1}
	{\rm dist}(0,\partial \widehat{F}(x))^\frac{1}{\alpha} \geq c_1(\widehat{F}(x)-\widehat{F}(\bar x)) =  c_1(\widehat{f} (x)-\widehat{f} (\bar x))\ \ \ \ \forall x \in \mathcal{M}\cap  B(\bar x, \epsilon_1).
	\end{equation}
	Since $x\mapsto {x}/{\|x\|}$ is continuous, there exists $\epsilon_2\in (0,1/2)$ such that when $x\in B(\bar x,\epsilon_2)$, it holds that ${x}/{\|x\|}\in B(\bar x,\epsilon_1)\cap \mathcal{M}$. Hence, for any $x\in B(\bar x,\epsilon_2)$, writing $u = {x}/{\|x\|}$ and using \eqref{distFG}, we see that
	\begin{equation*}\label{gjr2}
	\begin{split}
	&{\rm dist}(0, \partial G_{\bar x}(x))^\frac{1}{\alpha} \ge \left(\frac{2}{3}\right)^\frac{1}{\alpha}  \inf\limits_{\zeta\in\partial \widehat{f}(u)} \|(I- uu^T)\zeta\|^\frac{1}{\alpha}  \stackrel{\rm (a)}= \left(\frac{2}{3}\right)^\frac{1}{\alpha}{\rm dist} (0, \partial \widehat{F}(u))^\frac{1}{\alpha}\\
	& \ge  \left(\frac{2}{3}\right)^\frac{1}{\alpha}c_1(\widehat{f}(u)-\widehat{f} (\bar x)) = \left(\frac{2}{3}\right)^\frac{1}{\alpha} c_1(G_{\bar x}(x)-G_{\bar x}(\bar x)),
	\end{split}
	\end{equation*}
	where (a) follows from \eqref{gjr0}, the second inequality holds because of \eqref{gjr1}, and the last equality holds thanks to the definitions of $u$ and $G_{\bar x}$. This proves item (i).
	
	Now we turn to item~(ii). Note that $ \bar x\neq 0 $ (see \eqref{domains}). Suppose that $ G_{\bar x} $ satisfies the KL property at $ \bar x $ with exponent $ \alpha $. Since $G_{\bar x}$ is continuous on its domain,  there exist $ c_2>0 $ and $ \epsilon_3 \in (0,1/2)$ such that
	\begin{equation}\label{gjr5}
	{\rm dist}(0,\partial G_{\bar x}(x))^\frac{1}{\alpha} \geq c_2(G_{\bar x}(x)-G_{\bar x}(\bar x))\ \ \ \ \ \ \forall x\in B(\bar x, \epsilon_3)\cap {\rm dom}\,\partial G_{\bar x}.
	\end{equation}
	Since $\epsilon_3<1/2$, we see from \eqref{domains} that  $B(\bar x, \epsilon_3)\cap {\rm dom}\,\partial G_{\bar x} = B(\bar x, \epsilon_3)$.  This together with \eqref{gjr5} yields that
	\begin{equation*}
	\begin{split}
	&{\rm dist}(0,\partial \widehat{F}(x))^\frac{1}{\alpha} \stackrel{{\rm (a)}}={\rm dist}(0,\partial G_{\bar x}(x))^\frac{1}{\alpha} \geq c_2(G_{\bar x}(x)-G_{\bar x}(\bar x)) = c_2(\widehat{F}(x)-\widehat{F}(\bar x))
	\end{split}
	\end{equation*}
	whenever $ x\in B(\bar x, \epsilon_3)\cap \mathcal{M}$,
	where (a) holds because of the fact that $x\in\mathcal{M}$, \eqref{partialG} and \eqref{gjr0}, and the last equality holds because $x\in B(\bar x, \epsilon_3)\cap \mathcal{M}\subseteq B(\bar x,1/2)\cap \mathcal{M}$. This proves item (ii).
\end{proof}

Our next theorem shows that as long as the optimal value of \eqref{Manpp} is not zero, the function $\widehat F$ in \eqref{F1} is always a KL function with exponent $1/2$.
\begin{theorem}[KL exponent of $ 1/2 $]\label{KL12}
	Suppose that the optimal value of \eqref{Manpp} is not 0 and let $\widehat{F}$ be defined in \eqref{F1}. Then $\widehat{F}$ is a KL function with exponent $ 1/2 $.
\end{theorem}
\begin{proof}
Fix any $\bar x\in {\rm dom}\,\partial\widehat F = {\cal M}$ (see \eqref{domains}) and let $G_{\bar x}$ be defined in \eqref{G}. Then $\|\bar x\|=1$. Define
	\[
	\widetilde G_{\bar x}(x) = \widetilde f_{\bar x}(x) + \delta_{B\left(\bar x,1/2\right)}(x)
	\ \ \ {\rm where}\ \ \
	\widetilde f_{\bar x}(x): = \|Y^Tx\|_1^2 - \frac{\|Y^T\bar x\|^2_1}{\|\bar x\|^2}\|x\|^2.
	\]
	Note that $\widetilde f_{\bar x}(x) = \min_{\sigma \in {\frak R}}\{Q_\sigma(x) + P_\sigma(x)\}$ for some quadratic functions $Q_\sigma$ and polyhedral functions $P_\sigma$ indexed by $\sigma$, where ${\frak R} := \{u\in \R^n:\; u_i \in \{1,-1\}\ \ \forall i\}$: specifically, for each $\sigma\in {\frak R}$, we define $P_\sigma$ as the indicator function of the set $\{x:\; \sigma\circ( Y^Tx)\ge 0\}$, and $Q_\sigma(x) := (\langle\sigma,Y^Tx\rangle)^2-\frac{\|Y^T\bar x\|^2_1}{\|\bar x\|^2}\|x\|^2$. One can then deduce from \cite[Corollary~5.2]{LiPo18} that $\widetilde f_{\bar x}$ is a KL function with exponent $1/2$.
	
	Next, for $x\in {\rm int}(B\left(\bar x,1/2\right))$, we have
	\[
	\partial  \widetilde G_{\bar x}(x)  = \partial  \widetilde f_{\bar x}(x) {\rm \ and \ }\widetilde G_{\bar x}(x) = \widetilde f_{\bar x}(x).
	\]
	Thus,  $ \widetilde G_{\bar x}$ satisfies the KL property with exponent $1/2$ at $\bar x$.
	Now, define
	\[
	\bar G_{\bar x}(x) = \frac{\|Y^Tx\|_1^2}{\|x\|^2} + \delta_{B(\bar x,1/2)}(x).
	\]
Then ${\rm dom}\,\partial \bar G_{\bar x} = B(\bar x,1/2)$ (see \cite[Corollary~10.9]{RoWe97}).
	If $0\not \in \partial \bar  G_{\bar x}(\bar x)$, using \cite[Lemma~2.1]{LiPo18},  it holds that $\bar G_{\bar x}$    has KL property at $\bar x$ with exponent $1/2$. If $0\in \partial \bar  G_{\bar x}(\bar x)$, combining \cite[Theorem~4.2]{ZeYuPo20} with the fact that $ \widetilde G_{\bar x}$ satisfies the KL property with exponent $1/2$ at $\bar x$, we know that $\bar G_{\bar x}$  also satisfies the KL property at $\bar x$ with exponent $1/2$. Using this together with the fact that $ G_{\bar x}(\bar x) = \widehat{F}(\bar x)>0$ (since the optimal value of \eqref{Manpp} is not $0$ by assumption) and \cite[Proposition~4.3]{ZeYuPo20}, we deduce  that $G_{\bar x}$ satisfies the KL property with exponent $1/2$ at $\bar x$. The desired result now follows from Lemma \ref{lem1}(ii).
\end{proof}

Now we can establish the local convergence rate of $\Alg$ for solving \eqref{Manpp} by following a similar derivation in \cite[Theorem~3.4]{Attouch10}, and we omit the proof here.
\begin{theorem}[Local linear rate]\label{conv12}
	Let $\{x^k\}$ be the sequence generated by $\Alg$ with $\phi(s)=s$ for solving \eqref{Manpp} and suppose that the optimal value of \eqref{Manpp} is not 0. Then $\{x^k\}$ converges linearly to a stationary point of \eqref{Manpp}.
\end{theorem}

\subsubsection{Finite convergence of $\Alg$ for \eqref{Manpp} in the presence of weak sharp minima}

We will show that the $\widehat{F}$ in \eqref{F1} satisfies the KL property with exponent $0$ at any point in a set of weak sharp minima ${\cal X}$ of \eqref{Manpp} and use this to derive finite convergence of the $\{x^k\}$ generated by $\Alg$ when its limit lies in ${\cal X}$. To this end, we first show that $\widehat{F}$ in \eqref{F1} is prox-regular at any point in ${\rm dom}\,\partial \widehat F = \mathcal M$ (see \eqref{domains}) with any $ \bar v\in\partial \widehat{F}(\bar x) $.
\begin{proposition}\label{ProxF}
	The $\widehat{F}$ defined in \eqref{F1} is prox-regular at any $\bar x\in {\rm dom}\,\partial \widehat F$ for any $\bar v\in \partial \widehat{F}(\bar x)$.
\end{proposition}
\begin{proof}
Recall from \eqref{domains} that ${\rm dom}\,\partial \widehat F = {\cal M}$. Noting that $\delta_\mathcal M$ is the composition of $\delta_{\{1\}}(\cdot)$ and the map $x\mapsto \|x\|$, and using  \cite[Example~10.24(b)]{RoWe97} and \cite[Exercise~10.26(b)]{RoWe97}, one can deduce that $\delta_\mathcal M$ is strongly amenable at any $\bar x\in \mathcal M$. We can then see that $\delta_{\mathcal M}$ is prox-regular at any $\bar x\in \mathcal M$ according to \cite[Proposition~13.32]{RoWe97}.
	
	Note that $\widehat{f}$ is prox-regular at any $\bar x$ for any $\bar v\in \partial \widehat{f}(\bar x)$ since $ \widehat{f}$ is convex (see \cite[Example~13.30]{RoWe97}). In addition, using \cite[Proposition 8.12]{RoWe97}, we have $\partial^\infty \widehat{f}(\bar x) = \{0\}$ at any $\bar x\in\mathcal M$. Combining these with the fact that $\delta_{\mathcal M}$ is prox-regular at any $\bar x\in \mathcal M$ and \cite[Theorem~3.2]{PoRo10}, we obtain the desired result.
\end{proof}

We next establish a {\em uniform} prox-regularity property of $\widehat{F}$. We first observe that for any $x\in {\rm dom}\,\partial \widehat F = \mathcal M$,
\[
\partial \widehat F(x)\supseteq \widehat\partial \widehat F(x) \supseteq \widehat\partial \widehat f(x) = Y\partial \|Y^Tx\|_1,
\]
where the second inclusion follows from \cite[Corollary~10.9]{RoWe97} and Lemma~\ref{deltaM}(i).
Since $\partial\|\cdot\|_1$ is contained in the unit infinity norm ball, we see that $Y\partial \|Y^Tx\|_1\subseteq B(0,\rho)$ whenever $\rho> \sum_{i=1}^p\|{\bm y}_i\|$, where ${\bm y}_i$ is the $i_{\rm th}$ column of $Y$; in particular, $\partial \widehat F(x)\cap B(0,\rho)\neq \emptyset$ for all such $\rho$. In the next theorem, we show that the constants involved in the definition of prox-regularity can be chosen uniformly locally around any $\bar x\in {\cal M}$ for any $v$ in the {\em nonempty} set $\partial \widehat F(x)\cap B(0,\rho)$.
\begin{theorem}[Uniform prox-regularity of $ \widehat F $]\label{VaPG}
	Let $\widehat{F}$ be defined in \eqref{F1} and $\bar x\in {\rm dom}\,\partial \widehat F$. Let ${\bm y}_i$ denote the $i_{\rm th}$ column of $Y$. Then there exist $\rho> \sum_{i=1}^p\|{\bm y}_i\|$, $\epsilon>0$ and $r\ge 0$ such that
	\[
	\widehat{F}(x')\ge \widehat{F}(x) + \langle v,x' - x\rangle - \frac{r}{2}\|x' - x\|^2
	\]
	whenever $x',\, x\in {\rm dom}\,\widehat F\cap B(\bar x,\epsilon)$ and $v\in\partial \widehat{F}(x)\cap B(0,\rho)$.
\end{theorem}
\begin{proof}
	Let $\eta>0$ be such that  $\eta>\sum_{i=1}^p\|{\bm y}_i\|$. Then $\partial \widehat{F}(\bar x)\cap B(0,\eta)\neq \emptyset$ according to the discussion preceding this theorem. For any $\bar v\in\partial \widehat{F}(\bar x)\cap B(0,\eta)$, since  $\widehat{F}$ is prox-regular at $\bar x $ relative to $\bar v$ thanks to Proposition~\ref{ProxF}, there exist $\epsilon_{\bar v}>0$ and $r_{\bar v}\ge 0$ such that
	\begin{align}\label{PR0}
	\widehat{F}(x')\ge \widehat{F}(x) + \langle v,x' - x\rangle - \frac{r_{\bar v}}{2}\|x' - x\|^2
	\end{align}
	whenever $\|x' - \bar x\|<\epsilon_{\bar v}$, $\|x - \bar x\|<\epsilon_{\bar v}$, $\widehat{F}(x) < \widehat{F}(\bar x) + \epsilon_{\bar v}$, and $\|v - \bar v\|<\epsilon_{\bar v}$ with $v\in\partial \widehat{F}(x)$. Since $\widehat{F}$ is continuous on its domain, by shrinking $\epsilon_{\bar v}$ if necessary, we have \eqref{PR0} holds
	whenever $x',\, x\in {\rm dom}\,\widehat F\cap B(\bar x,\epsilon_{\bar v})$ and $v\in \partial \widehat{F}(x)\cap B(\bar v,\epsilon_{\bar v})$.
	
	Since $\partial \widehat{F}(\bar x)\cap B(0,\eta)$ is compact and covered by $ \bigcup_{\bar v\in\partial \widehat{F}(\bar x)\cap B(0,\eta)} {\rm int}(B(\bar v,\epsilon_{\bar v}))$, there exist $\{B(\bar v_i,\epsilon_{\bar v_i})\}_{i=1}^k$ with $\{\bar v_1,\dots,\bar v_k\}\subseteq \partial \widehat{F}(\bar x) \cap B(0,\eta)$ and $0<\epsilon_1<\min\{\epsilon_{\bar v_1},\dots,\epsilon_{\bar v_k},\eta - \sum_{i=1}^p\|{\bm y}_i\|\}$ such that
	\begin{align}\label{sub1}
	\partial \widehat{F}(\bar x)\cap B(0,\eta) + B(0,\epsilon_1)\subseteq \bigcup_{i=1}^k {\rm int}(B(\bar v_i,\epsilon_{\bar v_i}))
\subseteq \bigcup_{i=1}^k B(\bar v_i,\epsilon_{\bar v_i}).
	\end{align}
	
	Since $\widehat{F}$ is finite at $\bar x$, using \cite[Proposition~8.7]{RoWe97}, we have that $\partial \widehat{F}$ is outer semi-continuous (osc)  at $\bar x$ with respect to $x\stackrel{\widehat{F}}{\to}\bar x$ and thus osc at $\bar x$ relative to ${\rm dom}\, \widehat{F}$ (since $\widehat F$ is continuous on its domain). This together with the definition of osc implies that $x\mapsto \partial \widehat{F}(x)\cap B(0,\eta)$ is osc at $\bar x$ with respect to ${\rm dom}\, \widehat{F}$. Thus, using \cite[Proposition~5.12]{RoWe97}, for  $\rho:=\eta - \epsilon_1>\sum_{i=1}^p\|{\bm y}_i\|$ and $\epsilon_1$, there exists $0<\epsilon<\epsilon_1$ such that  for any $x\in B(\bar x,\epsilon)\cap{\rm dom}\,\widehat{F}$, it holds that
	\[
	\partial \widehat{F}(x)\cap  B(0,\rho)\subseteq \partial \widehat{F}(\bar x)\cap B(0,\eta) +  B(0,\epsilon_1)\subseteq \bigcup_{i=1}^k B(\bar v_i,\epsilon_{\bar v_i}),
	\]
	where the second inclusion makes use of \eqref{sub1}. Therefore, fixing any $x\in B(\bar x,\epsilon)\cap{\rm dom}\,\widehat{F}$ and $\widetilde v\in  \partial \widehat{F}(x)\cap  B(0,\rho)$, there exists $i_{\widetilde v}\in\{1,\dots,k\}$ such that $\widetilde v\in B(\bar v_{i_{\widetilde v}},\epsilon_{\bar v_{i_{\widetilde v}}})$. Using \eqref{PR0} with $\bar v = \bar v_{i_{\widetilde v}}$, we deduce that
	\begin{align*}
	\widehat{F}(x')\ge \widehat{F}(x) + \langle \widetilde v,x' - x\rangle - \frac{r_{v_{i_{\widetilde v}}}}{2}\|x' - x\|^2\ge \widehat{F}(x) + \langle \widetilde v,x' - x\rangle - \frac{r}{2}\|x' - x\|^2
	\end{align*}
	whenever $x'\in B(\bar x,\epsilon)\cap{\rm dom}\,\widehat{F}$, where  $r := \max\{r_{\bar v_1},\dots,r_{\bar v_k}\}$.  Since $x$ and $\widetilde v$ are arbitrarily chosen from $B(\bar x,\epsilon)\cap{\rm dom}\,\widehat{F}$ and $  \partial \widehat{F}(x)\cap B(0,\rho)$, the proof is complete in view of the above display.
\end{proof}

Now we are ready to prove that $\widehat{F}$ satisfies the KL property with exponent $0$ at any point in a set of weak sharp minima of $\widehat{F}$.
\begin{theorem}[Exponent 0 in weak sharp minima]\label{KL0}
	Let  $\widehat{F}$ be defined in \eqref{F1}. Suppose there exists a set $\mathcal{X}\subseteq {\cal M}$ of weak sharp minima for $\widehat{F}$ with parameter $(\alpha,\delta)$ for some $\alpha>0$ and $\delta>0$. Then $\widehat{F}$ satisfies the KL property  with exponent $0$ at any point in  $\mathcal{X}$.
\end{theorem}
\begin{proof}
Since $\widehat F$ is continuous on its domain, the closure of a set of weak sharp minima is also a set of weak sharp minima with the same parameter $(\alpha,\delta)$. Thus, by replacing ${\cal X}$ with its closure, we assume without loss of generality that ${\cal X}$ is closed. Moreover, we have ${\cal X}\subseteq {\rm dom}\,\partial\widehat F$ thanks to \eqref{domains}.

	Fix any $\bar x\in \mathcal{X}$.
	Using Proposition \ref{VaPG}, there exist $\rho>\sum_{i=1}^p\|{\bm y}_i\|$, $r\ge 0$ and $\epsilon\in (0,2 \delta)$ such that
	\begin{align}\label{UniPR}
	\widehat{F}(x')\ge \widehat{F}(x) + \langle v,x' - x\rangle - \frac{r}{2}\|x' - x\|^2
	\end{align}
	whenever $\|x' - \bar x\|\le\epsilon$ and $\|x - \bar x\|\le \epsilon$ with $\{x',x\}\subseteq \mathcal M$, and $v\in\partial \widehat{F}(x)\cap B(0,\rho)$.

Next, for any $x\in B(\bar x,\epsilon/2)$, let ${\rm Proj}_{\cal X}(x)$ be the collection of points in ${\cal X}$ that are closest to $x$: this set is nonempty thanks to the closedness of ${\cal X}$. Let $\bar x^* \in {\rm Proj}_{\cal X}(x)$. Then it holds that
\begin{equation}\label{distproj}
	\|\bar x^* - \bar x\|\le \|\bar x^* - x\| + \|x - \bar x\| = {\rm dist}(x,\mathcal{X}) + \|x - \bar x\|\le  2 \|x - \bar x\| \le  \epsilon.
	\end{equation}

Now, since $\epsilon< 2\delta$,  we have that $x\in \mathfrak{B}(\delta)$  whenever  $\|x - \bar x\|\le \epsilon/2$, where $\mathfrak{B}(\cdot)$ is defined in Definition \ref{sharpdef}. Thus, the weak sharpness of   $\mathcal{X}$ gives
	\begin{align}\label{sharp}
	\alpha\, {\rm dist}(x,\mathcal{X})\le  \widehat{F}(x) - \widehat{F}(\bar x^*)\ \ \ \ \forall  x \in B(\bar x,\epsilon/2).
	\end{align}
We can then deduce that for any $x\in\mathcal M\cap B(\bar x,\epsilon/2)$ and $v\in\partial \widehat{F}(x)\cap B(0,\rho)$, we have
	\begin{align*}
	\alpha\, {\rm dist}(x,\mathcal{X})\le  \widehat{F}(x) - \widehat{F}(\bar x^*)\le - \langle v,\bar x^* - x\rangle  + \frac{r}{2}\|\bar x^* - x\|^2\le \|v\|\|\bar x^* - x\| + \frac{r}{2}\|\bar x^* - x\|^2,
	\end{align*}
where the first inequality follows from \eqref{sharp}, and the second inequality follows from \eqref{UniPR} and \eqref{distproj}.
	Taking infimum over $v\in\partial \widehat{F}(x)\cap B(0,\rho)$, the above inequality gives that
	\begin{align}\label{dis}
	\alpha\, {\rm dist}(x,\mathcal{X})\!\le\!  {\rm dist}(0,\partial \widehat{F}(x)\cap B(0,\rho) )\|\bar x^* - x\| \!+\! \frac{r}{2}\|\bar x^* - x\|^2\ \ \ \ \forall x\in\mathcal M\cap B(\bar x,\epsilon/2).\!\!
	\end{align}
	
	On the other hand, invoking \eqref{partialF} (and recalling that $\phi(s)=s$), we see that
	\[
	Y u \in \partial \widehat{F}(x)\ \ \ \ \ \ \forall x\in\mathcal{M}{\rm \ and \ }u\in \partial\|\cdot\|_1(Y^Tx).
	\]
	Since $\partial\|\cdot\|_1$ is contained in the unit infinity norm ball, we have that $\|Yu\|\le \sum_{i=1}^p\|{\bm y}_i\|<\rho$, which  implies  that
	\[
	{\rm dist} (0, \partial \widehat{F}(x)) = {\rm dist} (0,\partial \widehat{F}(x)\cap B(0,\rho))\ \ \ \ \ \  \forall x\in\mathcal{M}.
	\]
	Combining this  with \eqref{dis}, we have  that
	\begin{equation*}
	\alpha\, {\rm dist}(x,\mathcal{X})\le  {\rm dist}(0,\partial \widehat{F}(x))\|\bar x^* - x\|  + \frac{r}{2}\|\bar x^* - x\|^2 = {\rm dist}(0,\partial \widehat{F}(x)) {\rm dist}(x,\mathcal{X}) + \frac{r}{2} {\rm dist}(x,\mathcal{X})^2,
	\end{equation*}
	for any $x\in\mathcal{M}\cap B(\bar x,\epsilon/2)$, where the equality follows from the definition of $\bar x^*$. Rearranging terms in the above inequality, we deduce further that
	\[
	\left(\alpha - \frac{r}{2}{\rm dist}(x,\mathcal{X})\right){\rm dist}(x,\mathcal{X})\le  {\rm dist}(0,\partial \widehat{F}(x)){\rm dist}(x,\mathcal{X})\ \ \ \ \forall x\in\mathcal{M}\cap B(\bar x,\epsilon/2).
	\]
	Let $\widetilde \delta \in(0, {2\alpha}/{r}) $ and $\widetilde \epsilon\in(0,\min\{\widetilde \delta,\epsilon/2\})$.  Then  for $x\in B(\bar x,\widetilde\epsilon)$, we have  ${\rm dist}(x,\mathcal{X})\le \|x - \bar x\| \le\widetilde \epsilon  <\widetilde \delta$  and  the above inequality further implies that
	\[
	\left(\alpha - \frac{r}{2}\widetilde \delta\right){\rm dist}(x,\mathcal{X})\le {\rm dist}(0,\partial \widehat{F}(x)){\rm dist}(x,\mathcal{X})
\ \ \ \ \ \ \ \forall x\in\mathcal{M}\cap B(\bar x,\widetilde \epsilon).
	\]
	Note that when $\widehat{F}(\bar x)<\widehat{F}(x)<\widehat{F}(\bar x) + 1$, we have $x\in\mathcal{M}\setminus \mathcal{X}$ and thus ${\rm dist}(x,\mathcal{X})>0$. Thus,  the above relation gives
	\[
	\alpha - \frac{r}{2}\widetilde \delta\le {\rm dist}(0,\partial \widehat{F}(x))
	\]
	whenever  $\|x - \bar x\|<\widetilde \epsilon$ and $\widehat{F}(\bar x)<\widehat{F}(x)<\widehat{F}(\bar x) + 1$. Since $\alpha - {r\widetilde \delta}/{2}>0$ thanks to the definition of $\widetilde \delta$, we conclude that $\widehat{F}$ satisfies the KL property with exponent $0$ at $\bar x$.
\end{proof}

Based on the KL exponent derived in Theorem~\ref{KL0}, we now show that if the sequence $\{x^k\}$ generated by $\Alg$ for solving \eqref{Manpp} has $\lim_{k\to\infty}x^k$ (which exists thanks to Theorem~\ref{thm1} and the fact that $F$ is semialgebraic, and hence a KL function, when $\phi(s)=s$) lying in a set of weak sharp minima, then $\{x^k\}$ converges finitely.
\begin{theorem}[Finite convergence]\label{convergefinite}
	Let $\mathcal{X}\subseteq {\cal M}$ be a set of weak sharp minima for \eqref{Manpp} with parameters $(\alpha,\delta)$.  Let $\{x^k\}$ be the sequence generated by $\Alg$ with $\phi(s)=s$. Suppose that $\lim_{k\to\infty}x^k=:x^*\in \mathcal{X}$.\footnote{The limit exists thanks to Theorem~\ref{thm1} and the fact that $F$ is semialgebraic (and hence a KL function) when $\phi(s)=s$.} Then $\{x^k\}$ converges finitely.
\end{theorem}
\begin{proof}
Recall that $\{\widehat F(x^k)\}$ is nonincreasing (see Theorem~\ref{sta}(i)) and bounded below (by zero). Hence, $\widehat{F}^*:= \lim_{k\to\infty}\widehat F(x^k)$ exists.
	We first claim that there exists $k_0\ge 0$ such that $\widehat{F}(x^{k_0}) = \widehat F^*$. Suppose to the contrary that $\widehat{F}(x^{k})> \widehat{F}^*$ for all $k\ge 0$. Since $\widehat{F}$ satisfies the KL property with exponent $0$ at $x^*$ thanks to Theorem \ref{KL0}, there exist $c>0$ and $a>0$ such that
	\[
	{\rm dist}(0,\partial \widehat{F}(x))\ge c,
	\]
	for $x\in B(x^*,a)\cap \{x:\widehat{F}(x^*)<\widehat{F}(x)<\widehat{F}(x^*) + a\}$. Since $\{x^k\}$ is convergent to $x^*$ and $\{\widehat{F}(x^{k})\}$ is continuous on $\mathcal{M}$, there exists an integer $N$ such that for all $k\ge N$, we have
	\[
	c\le {\rm dist}(0,\partial \widehat{F}(x^{k+1}))\le \left(C +\frac1t\right)\|d^k\|,
	\]
	where the second inequality follows from Lemma \ref{distbound}. Passing to the limit as $k$ goes to infinity in the above display and recalling Theorem \ref{sta}(ii), we have that $c\le 0$, a contradiction. Thus,  there exists $k_0\ge 0$ such that $\widehat{F}(x^{k_0}) = \widehat{F}^*= \lim_{k\to\infty}\widehat F(x^k)$. This together with Theorem~\ref{sta}(i) shows that $d^{k}\equiv 0 $
	for $k\ge k_0$ and thus $x^{k+1}\equiv x^k$ for $k\ge k_0$, i.e., $\{x^k\}$ converges finitely.
\end{proof}

\subsection{Relating KL exponents of $F$ in \eqref{F} and $\widetilde F$ in \eqref{Ftilde}}
In this subsection, we are interested in the connections between the KL exponents of $F$ in \eqref{F} and that of $\widetilde F$ in \eqref{Ftilde}, which is the potential function for establishing the global convergence of $\Alg$ for general $\phi$. We first present a theorem that deduces the KL exponent of $F$ from the KL exponent of $\widetilde F$. The proof of this theorem is similar to that of \cite[Theorem~5.1]{YuPongLi22}.
\begin{theorem}\label{KLrule1}
	Let $ F $ and $ \widetilde F $ be defined in \eqref{F} and \eqref{Ftilde}, respectively. Suppose that $ \widetilde F $ is a KL function with exponent $ \alpha\in [0,1) $. Then $ F $ is a KL function with exponent $ \alpha $.
\end{theorem}
\begin{proof}
	Let $ \bar x\in {\rm dom}\,\partial F=\mathcal{M} $ (see Lemma~\ref{prop:partialF}). Note from the definition of $ \widetilde F $ that
	\[
	\inf\limits_u\widetilde F(\bar x,u)  = \inf\limits_{-u\in {\rm dom}\, \Psi^*} \{\Psi^*(-u)+\langle u, |Y^T \bar x|\rangle \} = -\Psi(|Y^T\bar x|)=F(\bar x),
	\]
	where the second equality holds thanks to the definition of Fenchel Conjugate and \cite[Theorem~12.2]{Ro70}. Moreover,  we have from \cite[Theorem~23.5]{Ro70} that
	\[
	\Argmin_{u}\widetilde{F}(\bar x, u) =  \{\bar u \},\ \ \ \ {\rm where}\ \bar u := -\nabla \Psi ( |Y^T \bar x | ).
	\]
	In particular, we have $\widetilde F(\bar x, \bar u) = F(\bar x)<\infty $, which implies that $(\bar x, \bar u)\in {\rm dom}\, \widetilde F$.  Now using Proposition~\ref{subdifftildeF}, it holds at $ (\bar x, \bar u)$ that
	\begin{equation}\label{levelbnd}
		\partial \widetilde F(\bar x, \bar u) = \begin{bmatrix}
			Y(\bar u\circ \partial \|Y^T \bar x\|_1) +\partial \delta_{\mathcal M}(\bar x) \\
			| Y^T\bar x|-\partial 	\Psi^*(-\bar u)
		\end{bmatrix} \ni \begin{bmatrix} Y(\bar u\circ \partial \|Y^T \bar x\|_1) +\partial \delta_{\mathcal M}(\bar x) \\ 0\end{bmatrix},
	\end{equation}
	where the last inclusion uses the definition of $\bar u$ and \cite[Theorem~23.5]{Ro70}. One can now see that
	\[
	\partial \widetilde F(\bar x, \bar u) \neq \emptyset.
	\]
	Hence, $ \widetilde F $ satisfies the KL property at $ (\bar x, \bar u) $ since $ \widetilde F $ is a KL function. Now, we can invoke \cite[Theorem~3.1]{YuPongLi22} to show that $ F $ also satisfies the KL property at $\bar x$ with exponent $ \alpha $ once we prove that $ \widetilde F $ is level-bounded in $ u $ locally uniformly in $ x $.
	
	To establish the required level-boundedness property, we suppose to the contrary that there exist $ x^*\in \mathcal M $, $ \beta \in \R$ and an unbounded sequence
	\[
	\{(x^k, u^k)\}\subseteq \{(x,u):\, x\in \mathcal M,\ \|x-x^*\|\le 1,\ \widetilde F(x,u)\le \beta  \}
	\]
	with $ \|u^k\| \rightarrow \infty $. Passing to a subsequence if necessary, we may assume without loss of generality that  $ x^k\rightarrow \widetilde x $ for some $ \widetilde x\in \mathcal{M}\cap \{x:\,\|x-x^*\|\le 1 \} $ and $ \frac{u^k}{\|u^k\|}\rightarrow d $ for some $ d$ with $\|d\|=1$. Next, we have from the definitions of $ \{(x^k,u^k)\} $ and the definition of $ \widetilde F $ that
	\[
	\beta \geq \widetilde F(x^k,u^k) = \Psi^*(-u^k)+\langle u^k, |Y^Tx^k|\rangle
	\]
	for all $ k $. It then follows that
	\[
	\frac{\beta}{\|u^k\|}\geq \frac{\Psi^*(-u^k)}{\|u^k\|}+\left\langle |Y^T x^k|, \frac{u^k}{\|u^k\|}\right\rangle.
	\]
	Passing to the limit inferior on both sides of the above inequality and using the definitions of $ \widetilde x $ and $ d $, we have
	\[
		0  \geq \langle |Y^T \widetilde x|, d\rangle +\liminf\limits_{k\rightarrow \infty} \frac{\Psi^*(-u^k)}{\|u^k\|}
		\overset{{\rm (a)}}{\geq}  \langle |Y^T \widetilde x |, d\rangle +(\Psi^*)^\infty (-d) \overset{{\rm (b)}}{=} \langle |Y^T \widetilde x |, d\rangle + \sup\limits_{s\in {\rm dom}\, \Psi}\{-\langle s, d\rangle \},
	\]
	where $ (\Psi^*)^\infty(s):=\liminf_{t\rightarrow\infty,\, s'\rightarrow s}\frac{\Psi^*(ts')}{t} $ is the recession function of $ \Psi^* $ and (a) uses \cite[Theorem~2.5.1]{Auslender03} and (b) follows from \cite[Theorem~2.5.4]{Auslender03}. Since $ {\rm dom}\, \Psi = \R^p $, it then follows from the above relation that $ d=0 $, which contradicts the fact that $ \|d\|=1 $. Thus, $\widetilde F $ is level-bounded in $ u $ locally uniformly in $ x $.	
\end{proof}

We now establish the converse implication of Theorem~\ref{KLrule1} under additional assumptions on the $\phi$ in \eqref{P0}.
\begin{theorem}\label{KLcalculus_H}
	Let $F $ and $\widetilde F$ be given in \eqref{F} and \eqref{Ftilde}, respectively. Suppose that $\phi$ is strictly concave and $\phi'_+:\R_+\to \R_+$ is locally Lipschitz continuous. If $F$ satisfies the KL property at $ \bar x\in\mathcal M $ with exponent $\alpha\in \left[\frac12,1\right)$, then $(\bar x,-\nabla\Psi(|Y^T\bar x|))\in {\rm dom}\,\partial \widetilde F$ and $\widetilde F$ satisfies the KL property at $(\bar x,-\nabla\Psi(|Y^T\bar x))$ with exponent $\alpha$.
\end{theorem}
\begin{proof}
	Let $\bar u=-\nabla \Psi( | Y^T\bar x|)$. Then we see from \cite[Theorem~23.5]{Ro70} that
	\begin{equation}\label{domaineq}
	|Y^T\bar x|\in \partial \Psi^*(-\bar u)\ \ {\rm and}\ \ \widetilde F\left(\bar x, \bar u\right) = F(\bar x) < \infty,
	\end{equation}
	which yield in particular that $-\bar u\in {\rm dom}\, \partial \Psi^*$ and $\left(\bar x,\bar u \right) \in {\rm dom}\, \widetilde F$. These observations together with Proposition~\ref{subdifftildeF} show that $\left(\bar x,\bar u\right)\in {\rm dom}\,\partial \widetilde F$.
	
	Since $\phi$ is strictly concave, we see from the definition of $\psi$ in \eqref{psi} that $\psi$ is strictly convex and hence $\Psi$ is strictly convex. Then $\Psi^*$ is essentially smooth thanks to \cite[Theorem~26.3]{Ro70}. Now using \cite[Theorem~26.1]{Ro70}, we see that the set ${\rm dom}\, \partial \Psi^*$ is open and
	\[
	{\rm dom}\, \partial \Psi^*={\rm int}\left({\rm dom}\, \Psi^*\right);
	\]
	moreover, the gradient $\nabla \Psi^*(w)$ exists at any $w\in {\rm dom}\, \partial \Psi^*$.	
	
	Since ${\rm dom}\, \partial \Psi^*$ is open and recall that $(\bar x, -\bar u)\in \mathcal M\times {\rm dom}\, \partial \Psi^*$, we see that there exists $\epsilon > 0$ such that $-u \in {\rm dom}\,\partial \Psi^*$ whenever $\|u - \bar u\| \le \epsilon$. Then, according to \cite[Theorem~25.5]{Ro70}, the function $u\mapsto \nabla \Psi^*(-u)$ is continuous on $\{u:\; \|u - \bar u\| \le \epsilon\}$. Hence, upon recalling the first relation in \eqref{domaineq} and shrinking $\epsilon$ further if necessary, we may assume that
	\begin{equation}\label{less1}
		\||Y^Tx|-\nabla \Psi^*(-u) \| <1
	\end{equation}
	whenever $\max\{\|x - \bar x\|,\|u - \bar u\|\}\le\epsilon$.

Next, since $\phi'_+$ is locally Lipschitz, recalling the definition of $\psi$ in \eqref{psi} and shrinking $\epsilon$ further if necessary, we may assume in addition that $\nabla\Psi$ is globally Lipschitz on a compact convex set $\Upsilon$ containing $ \{ |Y^Tx |:\; \|x- \bar x\|\le \epsilon \}\cup\{\nabla \Psi^*(-u):\; \|u - \bar u\|\le\epsilon\}$.
	Then we have for any $(x,u)$ satisfying $x\in {\cal M}$  and $\max \{\|x - \bar x\|, \|u - \bar u \| \}\le\epsilon$ (which implies $-u\in {\rm dom}\,\partial\Psi^*$) that
	\begin{align}\label{valtilde}
			\widetilde F(x,u) &= \langle u, |Y^Tx|\rangle + \Psi^*(-u) =  F(x) + \Psi(|Y^Tx|) + \langle u, |Y^Tx|\rangle + \Psi^*(-u)\notag\\
			& = F(x) + \Psi(|Y^Tx|) - \Psi(\nabla \Psi^*(-u)) +\langle u,  |Y^Tx|-\nabla \Psi^*(-u)\rangle \notag\\
			&\le F(x) + \frac{\gamma}{2}\||Y^Tx|-\nabla \Psi^*(-u)\|^2,
	\end{align}
	where the third equality follows from the equality case for the Young's inequality (see \cite[Theorem~23.5]{Ro70}), and the inequality follows from the Lipschitz continuity of $\nabla\Psi$ on $\Upsilon$, where we let $\gamma$ denote the Lipschitz continuity modulus.
	
	Now, recall that $F$ satisfies the KL property with exponent $\alpha$ at $\bar x\in {\rm dom}\,\partial F = \mathcal M$. Thus, by shrinking $\epsilon$ further if necessary, there exist $c \in (0,1)$ such that
	\begin{equation}\label{2eq0}
		{\rm dist}(0,\partial F(x))^\frac1\alpha\ge c(F(x) - F(\bar x))
	\end{equation}
	whenever $ x\in \mathcal M $, $\|x - \bar x\|\le\epsilon$ and $F(x) < F(\bar x)+\epsilon$. In addition, since $\Psi$ has locally Lipschitz gradients, we see that $\partial \Psi^*$ is metrically regular at $-\bar u$ for $|Y^T\bar x|$; see \cite[Theorem~9.43]{RoWe97}. Shrinking $\epsilon$ and $c$ if necessary, we further have for every $(x,u)$ satisfying $\max\{\|x - \bar x\|,\|u - \bar u\|\}\le\epsilon$ (which implies $-u\in {\rm dom}\,\partial\Psi^*$) that
	\begin{equation}\label{metricreg}
		c\|u + \nabla \Psi(|Y^Tx|)\| \le \||Y^Tx| - \nabla \Psi^*(-u)\|.
	\end{equation}
	Furthermore, we note from the definition of $ \partial\|\cdot\|_1 $ that
	\begin{equation}\label{2eq88}
		\sup_{\|x - \bar x\|\le \epsilon}\sup_{y\in \partial \|Y^Tx\|_1}\|y\|_\infty \le 1,
	\end{equation}
	and there exists $M \ge 1$ such that
	\begin{equation}\label{2eq99}
		\sup_{\{x:\;\|x - \bar x\|\le \epsilon\}\cap \mathcal M}\|(I-xx^T)Y\| _2 \leq M.
	\end{equation}
	
	Finally, consider any $(x,u)$ satisfying $x\in {\cal M}$, $\max\{\|x - \bar x\|,\|u -\bar u\|\}\le\epsilon$ and
	\[
	\widetilde F(\bar x,\bar u) \le \widetilde F(x,u) < \widetilde F(\bar x,\bar u) + \epsilon.
	\]
	Then we have
	\begin{equation*}
		-u\in {\rm dom}\,\partial \Psi^*,\ \ (x,u)\in {\rm dom}\,\widetilde F\ \ {\rm and}\ \ F(\bar x) +\epsilon = \widetilde F(\bar x,\bar u) + \epsilon > \widetilde F(x,u)\ge F(x),
	\end{equation*}
	where the equality comes from \eqref{domaineq} and the last inequality follows from \eqref{infwF}. Moreover, we deduce from the two inclusions in the above display and Proposition~\ref{subdifftildeF} that
	\[
	\partial \widetilde F(x,u) = \begin{bmatrix}
		Y (u\circ \partial  \|Y^Tx \|_1 ) + \partial \delta_{\cal M}(x)\\
		|Y^Tx | - \partial \Psi^*(-u)
	\end{bmatrix}\neq \emptyset.
	\]
	For any such $(x,u)$, with $\mu := \sup\{3 |a|^\frac1\alpha + 3|b|^\frac1\alpha:\; a^2 + b^2 \le 1\} < \infty$, we have
	\begin{align*}
		&\mu\,{\rm dist}(0,\partial \widetilde F(x,u))^\frac1\alpha  = \mu\left(\sqrt{\inf\limits_{y\in\partial \left\|Y^Tx\right\|_1}\inf\limits_{\lambda\in\R} \left\|Y(u\circ y)+\lambda x\right\|^2 + \||Y^Tx| - \partial \Psi^*(-u)\|^2}\right)^\frac1\alpha\\
		& \overset{\rm (a)}= \mu\left(\sqrt{\left\|(I-xx^T)Y(u\circ \xi)\right\|^2 + \||Y^Tx| - \partial \Psi^*(-u)\|^2}\right)^\frac1\alpha\\
		& \geq 3 \|(I-xx^T)Y(u\circ \xi) \|^\frac1\alpha +3\||Y^Tx| - \nabla \Psi^*(-u)\|^\frac1\alpha\\
		& \overset{\rm (b)} \ge 2^{\frac1\alpha-1}\left ( \|(I-xx^T)Y(u\circ \xi) \|^\frac1\alpha + \| |Y^Tx | - \nabla \Psi^*(-u) \|^\frac1\alpha\right )+ \| |Y^Tx | - \nabla \Psi^*(-u) \|^\frac1\alpha\\
		& \overset{\rm (c)} \ge c^\frac1\alpha\left[2^{\frac1\alpha-1}\,\|(I-xx^T)Y(u\circ \xi)\|^\frac1\alpha + 2^{\frac1\alpha-1}\|u + \nabla \Psi(|Y^Tx|)\|^\frac1\alpha+\||Y^Tx| - \nabla \Psi^*(-u)\|^\frac1\alpha\right]\\
		& \overset{\rm (d)} \ge \left(cM^{-1}\right)^\frac1\alpha\bigg[2^{\frac1\alpha-1}\|(I-xx^T)Y(u\circ \xi)\|^\frac1\alpha\\
		&\quad\quad + 2^{\frac1\alpha-1}\|(I-xx^T)Y([\Phi'_+(|Y^Tx|)-u]\circ\xi)\|^\frac1\alpha +\||Y^Tx| - \nabla \Psi^*(-u)\|^\frac1\alpha\bigg]\\
		& \overset{\rm (e)} \ge \left(cM^{-1}\right)^\frac1\alpha\bigg[\|(I-xx^T)Y(\Phi'_+(|Y^Tx|)\circ \xi)\|^\frac1\alpha+\||Y^Tx| - \nabla \Psi^*(-u)\|^\frac1\alpha\bigg]\\
		& \overset{\rm (f)} \ge \left(cM^{-1}\right)^\frac1\alpha\left[{\rm dist}(0,\partial F(x))^\frac1\alpha + \||Y^Tx| - \nabla \Psi^*(-u)\|^\frac1\alpha\right] \\
		& \overset{\rm (g)}\ge c'\left[\frac1{c}{\rm dist}(0,\partial F(x))^\frac1\alpha + \frac\gamma{2}\||Y^Tx| - \nabla \Psi^*(-u)\|^\frac1\alpha\right]\\
		&\overset{\rm (h)}\ge c'\left[F(x) - F(\bar x) + \frac\gamma{2}\||Y^Tx| - \nabla \Psi^*(-u)\|^\frac1\alpha\right]\\
		& \overset{\rm (i)}\ge c'\left[F(x) - F(\bar x) + \frac\gamma{2}\||Y^Tx| - \nabla \Psi^*(-u)\|^2\right]\\
		& \ge c'\left[\widetilde F(x,u) - F(\bar x)\right] = c'\left[\widetilde F(x,u) - \widetilde F(\bar x,\bar u)\right],
	\end{align*}	
	where (a) follows from Lemma~\ref{prop:partialF} with
\[
 \xi\in \argmin_{y\in \partial  \|Y^Tx \|_1}  \|(I-xx^T)Y(u\circ y) \|,
 \]
 (b) holds because $\alpha\in \left [1/2, 1\right)$, (c) holds thanks to \eqref{metricreg} and the fact that $c\in (0,1)$, (d) uses \eqref{2eq88}, \eqref{2eq99} and the definition of $\Psi$ (see \eqref{psi}), (e) holds thanks the fact that $2^{\frac1\alpha -1}(\|a\|^{\frac1\alpha}+\|b\|^{\frac 1\alpha})\geq \|a+b\|^\frac 1\alpha $ for $\alpha<1$,  (f) follows Lemma~\ref{prop:partialF}, (g) holds with $c^\prime:= \big(cM^{-1}\big)^{1/\alpha}/\left(\frac1c+\frac \gamma 2\right)$, (h) follows from \eqref{2eq0}, (i) holds because of \eqref{less1} and the fact that $\alpha \in [1/2,1)$, and the last inequality follows from \eqref{valtilde} while the last equality comes from \eqref{domaineq}. Thus, $\widetilde F$ satisfies the KL property at $(\bar x,\bar u) = (\bar x, -\nabla \Psi(|Y^Tx|))$ with exponent $\alpha$.
\end{proof}

 \newpage

\noindent PEIRAN YU \\
Department of Computer Science \\
The University of Maryland\\
E-mail address: pyu123@umd.edu

\vspace{3mm}

\noindent LIAOYUAN ZENG\\
School of Mathematical Sciences\\
Zhejiang University of Technology\\
E-mail address: zengly@zjut.edu.cn

\vspace{3mm}
\noindent TING KEI PONG\\
Department of Applied Mathematics\\
The Hong Kong Polytechnic University\\
E-mail address: tk.pong@polyu.edu.hk

 \end{document}